\documentclass[12pt]{amsart}
\usepackage{graphics,wrapfig, graphicx, mathrsfs,xcolor}
\usepackage{mathtools}
\usepackage{amssymb,amscd,graphicx,xcolor,subfig,tikz}
\usepackage{graphics}
\usepackage{tikz}
\usepackage{indentfirst}
\usepackage{bm, enumerate,xcolor}
\usepackage[dvips]{epsfig}
\usepackage{latexsym}
\usepackage{float}
\usepackage[colorlinks]{hyperref}
\AtBeginDocument{
   \hypersetup{
    linkcolor=blue,
    citecolor=blue,
 }
}

\usepackage{amsmath}
\usepackage{amssymb}
\usepackage[applemac]{inputenc}
\usepackage[T1]{fontenc}
\newtheorem{lemma}{Lemma}[section]
\newtheorem{theorem}{Theorem}[section]
\newtheorem{corollary}{Corollary}[section]

\newtheorem{definition}{Definition}[section]
\newtheorem{remark}{Remark}[section]

\newtheorem{example}{Example}[section]

\numberwithin{equation}{section} \numberwithin{theorem}{section}
\numberwithin{example}{section} \numberwithin{remark}{section}
\numberwithin{figure}{section} \numberwithin{algorithm}{section}
\mathtoolsset{showonlyrefs}
\def\ds{\displaystyle}

\title[Boundary regularity in a Lipschitz domain]{Boundary regularity theory of the singular Lane-Emden-Fowler equation in a Lipschitz domain}
\author{Yahong Guo}
\address{School of Mathematical Sciences, Shanghai Jiao Tong University, Shanghai 200240, China}\email{yhguo@sjtu.edu.cn}
\author{Congming Li}
\address{School of Mathematical Sciences and CMA-Shanghai, Shanghai Jiao Tong University, China}\email{congming.li@sjtu.edu.cn}
\author{Chilin Zhang$^{\ast}$}\thanks{*Corresponding author}
\address{School of Mathematical Sciences, Fudan University, Shanghai 200433, China}\email{zhangchilin@fudan.edu.cn}

\begin{document}

\begin{abstract}
We study the singular Lane-Emden-Fowler equation
\begin{equation*}
    -\Delta u=f(X)\cdot u^{-\gamma}
\end{equation*}
in a bounded Lipschitz domain $\Omega$, with the Dirichlet boundary condition and a positive, bounded function $f(X)$. A distinguishing feature is that the vanishing boundary condition introduces a singularity in the equation.

We focus on the well-posedness of the equation and the growth rate of solutions near the boundary. The key is to classify the limiting cone of a boundary point into three categories based on its "frequency" and the scaling-invariant exponent  of the above equation, and to obtain distinct growth rate estimates for each case.

Additionally, we discuss the boundary Harnack principle for the singular Lane-Emden-Fowler equation, which is essential in deriving the boundary growth rate estimate. To our knowledge, the boundary Harnack principle we derive is the first Kemper-type estimate for singular semi-linear equations. It notably differs from the classical one for linear equations, in particular, the boundedness of the ratio \(u/v\) does not imply its continuity.

To address the lack of a suitable upper barrier, we introduce new techniques, including constructing upper barriers iteratively. We also construct a subharmonic auxiliary function $V(X)$ related to the solution $u$ in the limiting cone. The growth rate of $u(X)$ is then obtained inductively from the growth rate of the auxiliary function $V(X)$. Our results and methods offer novel insights into the behavior of singular elliptic equations in non-smooth domains.
\end{abstract}
\maketitle

\section{Introduction}
\subsection{Background}
In this paper, we investigate the asymptotic boundary behavior of non-negative solutions to the singular Lane-Emden-Fowler equation:
\begin{equation}\label{eq. main -1}
    \left\{\begin{aligned}
        &-\Delta u=f(X)u^{-\gamma}&\mbox{ in }&\Omega,\\
        &u=0&\mbox{ on }&\partial\Omega,
    \end{aligned}
    \right.
\end{equation}
where  $\gamma>0$,  $\Omega$ is  an open Lipschitz domain (either bounded or unbounded), and $f\in C^{\alpha}_{loc}(\Omega)(0<\alpha<1)$ is a positive and bounded function. For brevity, we refer to this equation, which exhibits a singularity when the solution $u$ approaches the zero boundary value,  as the "SLEF"  throughout this paper.

The "SLEF" is a natural generalization of the classical Lane-Emden equation
\begin{equation*}
    -\Delta u=u^{p}\quad(p>0),
\end{equation*}
which models numerous phenomena in mathematical physics and astrophysics. When in particular $p=\frac{n+2}{n-2}$, the problem becomes the Nirenberg problem, describing the conformal mapping of a Euclidean domain to a manifold with constant positive scalar curvature. The boundary regularity of solutions to the classical Lane-Emden equation in a smooth domain has been studied in \cite{CFRS20,CLM20,GT01,GLO26,KS18,QS19,RW94}. For a punctured domain, the asymptotic behavior near the isolated singularity is studied in the pioneering work \cite{CGS89}, see also \cite{HLL21,KMPS99,L96}.

We also highlight the well-known Loewner-Nirenberg problem \cite{LN74}, which, in our view, exhibits significant similarities to the singular Lane-Emden-Fowler equation. This problem involves studying the negative scalar curvature equation
\begin{equation}\label{eq. Loewner-Nirenberg}
    \Delta u=u^{\frac{n+2}{n-2}}\mbox{ in }\Omega,\quad u=\infty\mbox{ on }\partial\Omega.
\end{equation}
where $\Omega\subseteq\mathbb{R}^{n}$ $(n\geq3)$ is a Lipschitz or $C^{1,\alpha}$ domain and $u\geq0$. Higher regularity estimates near the boundary were later established in \cite{ACF92,K05,M91} under better regularity assumptions on $\partial\Omega$. More recently, \cite{HJS24,HS20,J21} derived an asymptotic expansion for solutions to \eqref{eq. Loewner-Nirenberg} in conic domains. Notably, prior to these works, the authors of \cite{HS18} obtained analogous estimates for the negative curvature Liouville equation. The readers may compare our present results with those for the Loewner-Nirenberg problem to identify both parallels and distinctions.

From a physical point of view, equations of the form \eqref{eq. main -1} represent a generalized version of the Lane-Emden-Fowler equation and find applications across multiple scientific domains. For instance, in the study of thermo-conductivity \cite{FuMa} where $u^\gamma$ models material resistivity, in the theory of gaseous dynamics in astrophysics \cite{Fo}, in signal transmission \cite{No}, and in the context of chemical heterogeneous catalysts \cite{Pe}. Among these applications, one of the most physically significant implementations occurs in the study of non-Newtonian pseudo-plastic fluids, where such singular equations model boundary layer phenomena (see e.g.,\cite{AcShpe,Naca, VasoMo} and references therein).

From a theoretical standpoint, following the pioneering existence and uniqueness result presented in \cite{FuMa}, a systematic study of problems such as \eqref{eq. main -1} in a bounded smooth domain $\Omega$ was initiated in \cite{CrRaTa,Stu}. If $f$ is smooth enough and bounded away from zero on $\Omega$, then the existence and uniqueness result of classical solutions $u\in C^{\alpha}(\overline{\Omega})\cap C^2(\Omega)$ with $\alpha=\frac{2}{1+\gamma}$ is established by applying an appropriate sub- and super-solution method, with significant refinements later provided in \cite{LaMc}. Specifically, when $f\in C^\alpha(\bar{\Omega})$ and $f>0$ in $\bar{\Omega}$, the authors first established the estimate for $u$ for $\gamma>1$:
\begin{equation}\label{eq. u is eigenfunction}
    a\phi^{\frac{2}{1+\gamma}}\leq u\leq b \phi^{\frac{2}{1+\gamma}},
\end{equation}
where $a,b>0$, and $\phi$ denotes the first eigenfunction of $-\Delta$ under the Dirichlet boundary condition. Then they showed that the solution $u\in W^{1,2}(\Omega)$ if and only if $\gamma<3$. This result has been extended to a more general datum $f\in L^p(\Omega)(p\geq 1)$ in \cite{OlPe}, where the authors established that the necessary and sufficient condition for the existence of the energy solution is \( \gamma<3-\frac{2}{p} \) (see also \cite{BoOr} for further insights). Later, del Pino \cite{dP} proved the existence of a unique solution $u\in C^{1,\alpha}({\Omega})\cap C(\bar{\Omega})$. When \( f(X) \) satisfies either
\begin{itemize}
    \item[]Case 1: $f(X)$ is merely nonnegative and bounded,
    \item[]Case 2: $f(X)$ behaves asymptotically as $f(X) \sim d(X, \partial\Omega)^\alpha$,
\end{itemize}
Gui and Lin \cite{GuLi} demonstrated that the positive solution is in general H\"older-continuous up to the boundary and  exhibits even higher regularity in certain special cases. For instance, if $\alpha-\gamma\geq 0$ in Case 1, then $u\in C^{1,\beta}(\bar{\Omega})$ for all \( \beta \in (0,1) \). Similarly,  
  if \( \gamma < 1 \) in Case 2, then \( u \in C^{1,1-\gamma}(\bar{\Omega}) \). They obtained such a regularity result by establishing the following sharp estimates for $u$ near the boundary (taking Case 1 as an example):
  \begin{equation}\label{BBS}
      u(X)\sim\left\{\begin{aligned}
          &d(X,\partial\Omega)^{\frac{2}{1+\gamma}},& \ \mbox{if}\ \gamma>1;\\ 
      &d(X,\partial\Omega)|\ln{d(X,\partial\Omega)}|^{\frac{1}{2}} ,& \ \mbox{if}\ \gamma=1;\\
      &d(X,\partial\Omega),& \ \mbox{if}\ \gamma<1.
      \end{aligned}\right.
  \end{equation}
We also recommend the readers to read the survey paper \cite{OP24}, in which Oliva and Petitta gave a systematic illustration of the theory mentioned above.

There are also interesting results in the case where $f(X)=1$. For instance, when $\Omega$ is a half-space, Montoro, Muglia
and Sciunzi recently classified positive solutions to \eqref{eq. main -1} in \cite{MMS24a,MMS24b}. Their work established a uniform asymptotic growth estimate near the boundary $\mathbb{R}^{n-1}$ for $\gamma>1$:
\[c_1x_n^{\frac{2}{1+\gamma}}\leq 
 u(X)\leq c_2x_n^{\frac{2}{1+\gamma}},\] which has been extended to the singular fractional Lane-Emden-Fowler equation in \cite{GuoWu}. When $\Omega$ is a quarter-space in $\mathbb{R}^{2}$ (e.g., $\Omega=\{x_{1},x_{2}\geq0\}$), the boundary regularity and the Krylov-type boundary Harnack principle were derived in \cite{HZ24}. Specifically, for $\gamma\in(0,1)$, any positive solution to $-\Delta u=u^{-\gamma}$ in $\Omega\cap B_{1}$ (vanishing at $\partial\Omega$) satisfies
 \begin{equation}\label{eq. HZ boundary Harnack}
     \frac{u(X)}{\Psi(X)}-1=O(|X|^{\epsilon})\quad\mbox{in }\Omega\cap B_{1/2},\mbox{ for some }\epsilon>0,
 \end{equation}
where $\Psi(X)$ is the homogeneous solution of degree $\frac{2}{1+\gamma}$ to the same equation.
 
However, for general non-smooth domains, such as those with Lipschitz boundaries, to the best of our knowledge, no growth rate results analogous to \eqref{eq. HZ boundary Harnack} near the boundary have been established for the singular problem \eqref{eq. main -1}. This gap motivates the present work, which investigates the boundary behavior of solutions to the singular Lane-Emden-Fowler equation \eqref{eq. main -1}. Specifically, we proceed as follows:
\begin{itemize}
    \item[(1)] Well-posedness: Using the continuity method, we first establish the well-posedness of "SLEF" in bounded Lipschitz domains in the classical sense.
    \item[(2)] Growth rate estimates: We then employ an iteration technique to derive almost sharp growth rate estimates near the boundary (see Theorem~\ref{thm. growth rate} and Theorem~\ref{thm. growth rate, (b) improved}), where the boundary Harnack principle plays a pivotal role.
    \item[(3)] The boundary Harnack principle: Additionally, we extend the boundary Harnack principle from linear equations to the semi-linear singular equation \eqref{eq. main -1} (see Theorem~\ref{thm. Harnack Comparable, general Lipschitz domain}).
\end{itemize}

In the process of deriving the boundary growth rate estimates in a Lipschitz domain, we realize that classical methods for a smooth domain fail. For example, the estimate \eqref{eq. u is eigenfunction} obtained in \cite{LaMc} no longer holds if $\partial\Omega$ has a $\frac{\pi}{2}$ angle, see the estimate \eqref{eq. HZ boundary Harnack}. Besides, if $\Omega$ is not $C^{1,1}$, then there are no interior and exterior tangential balls used in \cite{GuLi} that provide the estimate \eqref{BBS}. In the case of a Lipschitz domain, parallel to \cite{GuLi}, we study interior and exterior cones rather than interior and exterior balls near a boundary point. For each exponent $\gamma$, we classify the cones into three categories based on the "frequency" (the degree of the homogeneous harmonic function in the cone), see Definition~\ref{def. cone classification}. Then we obtain essentially different growth rate estimates in each case.
 
To address the challenges arising from both the singular nonlinear term and the low regularity of the domain, we introduce several novel techniques, including:
\begin{itemize}
    \item[(1)] In proving the growth rate estimates in Theorem~\ref{thm. growth rate}, we leverage the rescaling invariance of the equation to construct a subharmonic function $V(X)$ in a Lipschitz Cone $ Cone_{\Sigma}$   that vanishes on the boundary:
    \begin{equation*}
        V(X)=2^{\frac{2}{1+\gamma}}U(X/2)-U(X),
    \end{equation*}
    which  captures the difference between the solution and its rescaled counterpart. Then by applying the classical  boundary Harnack principle to this auxiliary 
 function $V(X)$ and employing an iterative approach, we are able to derive the growth rate of solutions  near the boundary. 
    \item[(2)] To obtain the optimal growth rate in the critical case (see  Theorem~\ref{thm. growth rate, (b) improved} below), we iteratively construct barrier functions to compensate for the absence of a global, suitable and explicitly defined barrier function within the Lipschitz cone $Cone_{\Sigma}$.
    \item[(3)] In developing the Kemper-type boundary Harnack principle for the "SLEF" equation, we go beyond the classical boundary Harnack principle for linear equations. After establishing the boundedness of the ratio $\frac{u}{v}$ of any two solutions, we  need to further explore its continuity   based on the size of the cone $Cone_{\Sigma}$(see Theorem 1.5-1.7 below). In particular, in the super-critical case ($\gamma$ small) with $C^1$ or convex boundary (see Theorem 1.7 below), we combine the boundary Harnack principle with the Campanato iteration to derive a pointwise Schauder-type estimate(see Theorem 1.8 below), which plays a crucial role for deriving the continuity.
\end{itemize}

\subsection{Definitions and notations}
We introduce the following fundamental definitions and notations that will be used throughout this work.

Let $\Sigma\subseteq\partial B_{1}$ be an open (relative to the topology of $\partial B_{1}$) spherical domain with a Lipschitz boundary. Its associated cone is denoted as
\begin{equation*}
    Cone_{\Sigma}:=\{X=tY:\ t>0,\ Y\in\Sigma\}.
\end{equation*}
Inspired by the Almgren's frequency function, we denote $\lambda_{\Sigma}$ to be the first Dirichlet eigenvalue of  the spherical Laplacian on $\Sigma$, and define $\phi_{\Sigma}>0$  to be the "frequency" of $Cone_{\Sigma}$. Precisely, we have
\begin{equation*}
    \lambda_{\Sigma}:=\min_{v\in H_{0}^{1}(\Sigma)}\frac{\int_{\Sigma}|\nabla v|^{2}d\theta}{\int_{\Sigma}v^{2}d\theta},\quad\phi_{\Sigma}(n+\phi_{\Sigma}-2)=\lambda_{\Sigma}.
\end{equation*}
It is worth mentioning that if $E_{\Sigma}(\theta)\in C_{0}^{\infty}(\Sigma)\cap C(\overline{\Sigma})$ is a normalized minimizer satisfying
\begin{equation*}
    \int_{\Sigma}|\nabla E_{\Sigma}|^{2}d\theta=\lambda_{\Sigma}\int_{\Sigma}E_{\Sigma}^{2}d\theta,\quad\max_{\theta\in\Sigma}E_{\Sigma}(\theta)=1,
\end{equation*}
then the function $H_{\Sigma}$, expressed in polar coordinates as 
\begin{equation}\label{eq. H Sigma, homogeneous harmonic in a cone}
    H_{\Sigma}=H_{\Sigma}(r,\theta)=r^{\phi_{\Sigma}}E_{\Sigma}(\theta)
\end{equation}
is a positive harmonic function supported in $Cone_{\Sigma}$.
\begin{example}\label{ex. half space cone}
    If $Cone_{\Sigma}=\mathbb{R}^{n}_{+}$ is a half space, then $\phi_{\Sigma}=1$ and
    \begin{equation*}
        H_{\Sigma}(X)=x_{n}.
    \end{equation*}
\end{example}
\begin{example}
    If $Cone_{\Sigma}\subseteq\mathbb{R}^{2}$ is an angle of size $\Theta$, then $\phi_{\Sigma}=\frac{\pi}{\Theta}$ and
    \begin{equation*}
        H_{\Sigma}(r,\theta)=r^{\frac{\pi}{\Theta}}\sin{(\frac{\theta}{\Theta}\pi)},\quad r>0,\ \theta\in(0,\Theta).
    \end{equation*}
\end{example}


Assume for simplicity that $u$ satisfies $-\Delta u=u^{-\gamma}$ in $B_{1}\cap Cone_{\Sigma}$ and vanishes on $\partial(B_{1}\cap Cone_{\Sigma})$. By comparing the super-harmonic function $u$ with the homogeneous harmonic function $H_{\Sigma}$ defined in \eqref{eq. H Sigma, homogeneous harmonic in a cone}, we find that $\ds\max_{B_{r}\cap Cone_{\Sigma}}u(x)\gtrsim r^{\phi_{\Sigma}}$. On the other hand, the $\frac{2}{1+\gamma}$ scaling invariance of the equation suggests that $\ds\max_{B_{r}\cap Cone_{\Sigma}}u(x)\gtrsim r^{\frac{2}{1+\gamma}}$. We believe that the competition between $\phi_{\Sigma}$ and $\frac{2}{1+\gamma}$ is the decisive factor governing the boundary behavior of the solution.

Motivated by this philosophy, we classify cones into three types by comparing its frequency with the scaling-invariant exponent $\frac{2}{1+\gamma}$ of the SLEF equation \eqref{eq. main -1}.
\begin{definition}(classification of cones with respect to $\gamma$)\label{def. cone classification}
    Let $\Sigma\subseteq\partial B_{1}$ be an open spherical domain with Lipschitz boundary, and let $\gamma>0$ be a given exponent appearing in \eqref{eq. main -1}. We say the cone $Cone_{\Sigma}$ is
    \begin{itemize}
        \item sub-critical (with respect to $\gamma$) if $\phi_{\Sigma}>\frac{2}{1+\gamma}$;
        \item critical (with respect to $\gamma$) if $\phi_{\Sigma}=\frac{2}{1+\gamma}$;
        \item super-critical (with respect to $\gamma$) if $\phi_{\Sigma}<\frac{2}{1+\gamma}$.
    \end{itemize}
The phrase "with respect to $\gamma$"  may be omitted when there is no risk of confusion.
\end{definition}
\begin{remark}
    If $Cone_{\Sigma}$ is the half space $\mathbb{R}^{n}_{+}$, we have already seen in Example~\ref{ex. half space cone} that $\phi_{\Sigma}=1$. In this case, the classification in Definition~\ref{def. cone classification} becomes the following:
    \begin{itemize}
        \item when $\gamma>1$, then $Cone_{\Sigma}$ is sub-critical (with respect to $\gamma$);
        \item when $\gamma=1$, then $Cone_{\Sigma}$ is critical (with respect to $\gamma$);
        \item when $\gamma<1$, then $Cone_{\Sigma}$ is super-critical (with respect to $\gamma$).
    \end{itemize}
    This corresponds to the trichotomy in \cite{GuLi}, see also \eqref{BBS}.
\end{remark}

A bounded domain $\Omega$ is called a Lipschitz domain, denoted by $\partial\Omega\in C^{0,1}$, if its boundary $\partial\Omega$ can be locally expressed as the  graph $\Gamma$ of  a Lipschitz function $g(x')$ near each boundary point $X\in\partial\Omega$ after a suitable rotation, that is, $$\Gamma:=\{X=(x',x_n)\in \partial \Omega:~ x_{n}=g(x')\}$$ with $[g(x')]_{C^{0,1}}\leq L$ for some $L$ (uniform on $\partial\Omega$). The minimal constant $L=:[\partial\Omega]_{C^{0,1}}$ is called the Lipschitz constant of the domain $\Omega$.


\begin{definition}(the local super-critical interior cone condition)\label{def. super-critical interior cone condition}
    Let $\Omega$ be a Lipschitz domain whose boundary $\Gamma$ passes through the origin. We say $\Omega$ satisfies "the super-critical interior cone condition" near the origin, if the following holds: there exists a super-critical Lipschitz cone $Cone_{\Sigma}$ and a radius $r>0$ such that for all $X\in\Gamma\cap B_{r}$, one has 
\begin{equation}\label{eq. super-critical interior cone condition}
    (X+Cone_{\Sigma})\cap B_{r}(X)\subseteq\Omega.
\end{equation}
\end{definition}
\begin{remark}
    For any $Cone_{\Sigma}$ satisfying \eqref{eq. super-critical interior cone condition} for all $X\in\Gamma\cap B_{r}$, the domain $\Sigma$ must be disjoint from its antipodal set $-\Sigma$. It follows that $|\Sigma| \leq \frac{1}{2}|\partial B_{1}|$,  and hence we have $\phi_{\Sigma}\geq1$ by the Schwartz symmetrization. Furthermore, $\gamma<1$ is a necessary condition of "the local super-critical interior cone condition". In fact, if the cone $Cone_{\Sigma}$ is super-critical (with respect to $\gamma$), i.e., $\frac{2}{1+\gamma}>\phi_{\Sigma}$, then $\gamma<1$.
\end{remark}
We finally introduce several neighborhoods that will be useful for deriving boundary regularity estimates.
\begin{definition}[cylindrical neighborhoods]
    For a boundary point $X=(x',g(x'))\in\Gamma$, we define three types of cylindrical neighborhoods:
    \begin{itemize}
        \item A grounded cylinder $\mathcal{GC}_{r}(X)$ as
        \begin{equation*}
            \mathcal{GC}_{r}(X):=\{Y=(y',y_{n}):\quad|x'-y'|\leq r,\quad0\leq y_{n}-g(y')\leq r\};
        \end{equation*}
        \item A doubled cylinder $\mathcal{DC}_{r}(X)$ as
        \begin{equation*}
            \mathcal{DC}_{r}(X):=\{Y=(y',y_{n}):\quad|x'-y'|\leq r,\quad|y_{n}-g(y')|\leq r\};
        \end{equation*}
        \item A suspended cylinder $\mathcal{SC}_{r,\delta}(X)$ as
        \begin{equation*}
            \mathcal{SC}_{r,\delta}(X):=\{Y=(y',y_{n}):\quad|x'-y'|\leq r,\quad\delta r\leq y_{n}-g(y')\leq r\}.
        \end{equation*}
    \end{itemize}
    The parameter $\delta\leq\frac{1}{10}$ will be specified later. When the reference point $X=0$ and parameter $\delta$ are clear from context, we use the simplified notations:
    \begin{equation*}
        \mathcal{GC}_{r}=\mathcal{GC}_{r}(0),\quad\mathcal{DC}_{r}=\mathcal{DC}_{r}(0),\quad\mathcal{SC}_{r}=\mathcal{SC}_{r,\delta}=\mathcal{SC}_{r,\delta}(0).
    \end{equation*}
\end{definition}

\subsection{Main results}
We present the principal contributions of this work, beginning with the existence and uniqueness theory for the singular problem.

\begin{theorem}[well-posedness]\label{thm. well-posedness}
    Let $\Omega$ be a bounded open Lipschitz domain. 
   Suppose $f(X)$ is a locally H\"older continuous function in $\Omega$ with $0<\lambda\leq f(X)\leq\Lambda$, and  $\varphi\geq0$ is a continuous function  on $\partial\Omega$. Then there exists a unique classical solution $u(X):\Omega\to\mathbb{R}_{+}$ satisfying:
    \begin{equation}\label{eq. Dirichlet}
        \left\{\begin{aligned}
            &-\Delta u(X)=f(X)\cdot u(X)^{-\gamma}&\mbox{ in }&\Omega\\
            &u(Y)=\varphi(Y)&\mbox{ on }&\partial\Omega
        \end{aligned}\right..
    \end{equation}
\end{theorem}
\begin{remark}
    Theorem~\ref{thm. well-posedness} remains valid under the weaker assumption that $f(X)$  belongs to the class $C^{Dini}_{loc}(\Omega)$.
\end{remark}

While interior regularity follows from the standard elliptic theory, the boundary behavior presents distinctive challenges due to the singular nature of the problem. The growth of the solution near vanishing boundary data depends on the geometry of the domain,  in particular on the "sharpness" of boundary points.

\begin{theorem}[growth rate estimate]\label{thm. growth rate}
    Let $\Gamma$ be the graph of $g$, such that $g(0)=0$ and $\|g\|_{C^{0,1}}\leq L$. Assume that $u$ satisfies
    \begin{equation}\label{eq. main}
    \left\{\begin{aligned}
        &-\Delta u=f(X)u^{-\gamma}&\mbox{ in }&\mathcal{GC}_{3R}\\
        &u=0&\mbox{ on }&\Gamma\cap\mathcal{GC}_{3R}
    \end{aligned}
    \right.
\end{equation}
    with $0<\lambda\leq f(X)\leq\Lambda$ being H\"older continuous.
    
    If $\mathcal{GC}_{3R}\subseteq\overline{Cone_{\Sigma}}$ for an open spherical domain $\Sigma$ with Lipschitz boundary, then for some $C=C(n,L,\gamma,\lambda,\Lambda,\phi_\Sigma)$, we have the following upper bound estimates:
    \begin{itemize}
        \item[(a)] If the cone $Cone_{\Sigma}$ is sub-critical, i.e. $\frac{2}{1+\gamma}<\phi_{\Sigma}$, then
        \begin{equation*}
            u(X)\leq CR^{\frac{-2}{1+\gamma}}\|u\|_{L^{\infty}(\mathcal{GC}_{3R})}\cdot|\frac{X}{R}|^{\frac{2}{1+\gamma}},\ \forall X\in\mathcal{GC}_{R} ;
        \end{equation*}
        \item[(b)] If the cone $Cone_{\Sigma}$ is critical, i.e. $\frac{2}{1+\gamma}=\phi_{\Sigma}$, then
        \begin{equation*}
            u(X)\leq CR^{\frac{-2}{1+\gamma}}\|u\|_{L^{\infty}(\mathcal{GC}_{3R})}\cdot|\frac{X}{R}|^{\phi_{\Sigma}}\ln{\frac{2R}{|X|}},\ \forall X\in\mathcal{GC}_{R} ;
        \end{equation*}
        \item[(c)] If the cone $Cone_{\Sigma}$ is super-critical, i.e. $\frac{2}{1+\gamma}>\phi_{\Sigma}$, then
        \begin{equation*}
            u(X)\leq CR^{\frac{-2}{1+\gamma}}\|u\|_{L^{\infty}(\mathcal{GC}_{3R})}\cdot|\frac{X}{R}|^{\phi_{\Sigma}},\ \forall X\in\mathcal{GC}_{R} ;
        \end{equation*}
    \end{itemize}
    
    If we assume instead that $\mathcal{GC}_{3R}\cap B_{r}\supseteq Cone_{\Sigma}\cap B_{r}$ for some small $r>0$, then we have the following lower bound estimates (along a ray):
    \begin{itemize}
        \item[(d)] If the cone $Cone_{\Sigma}$ is sub-critical, then
        \begin{equation*}
            u(t\vec{e_{n}})\geq c t^{\frac{2}{1+\gamma}}\mbox{ for }t\in[0,\epsilon]
        \end{equation*}
        for some $c,\epsilon>0$. If the cone $Cone_{\Sigma}$ is super-critical, then
        \begin{equation*}
            u(t\vec{e_{n}})\geq c t^{\phi_{\Sigma}}\mbox{ for }t\in[0,\epsilon].
        \end{equation*}
        \item[(e)] If the cone $Cone_{\Sigma}$ is critical, then
        \begin{equation*}
            u(t\vec{e_{n}})\geq c t^{\phi_{\Sigma}}(\ln{\frac{1}{t}})^{\phi_{\Sigma}/2}\mbox{ for }t\in[0,\epsilon].
        \end{equation*}
    \end{itemize}
\end{theorem}
\begin{remark}
    In fact, we only need to assume $f(X)\leq\Lambda$ in (a)-(c), and  $f(X)\geq\lambda$ in (d)(e).
\end{remark}
\begin{remark}
    In particular, if the region $\mathcal{GC}_{3R}$ coincides with the cone $Cone_{\Sigma}$ near the origin, then the estimates (a)(c)(d) in Theorem~\ref{thm. growth rate} are optimal. For the critical cone, however, we have $\displaystyle\frac{\phi_{\Sigma}}{2}=\frac{1}{1+\gamma}<1$, which creates a gap between the growth rates in (b) and (e).
\end{remark}
The precise growth estimates in Theorem~\ref{thm. growth rate} yield several important regularity consequences, which we summarize in the following corollary (its proof will be postponed until after the proof of Theorem~\ref{thm. growth rate}).
\begin{corollary}\label{cor. of thm 1.2}
Assume that $\Omega$ is a bounded Lipschitz domain and $0<\lambda\leq f(X)\leq\Lambda$. Then the solution $u$ to \eqref{eq. main -1} is $C^{\mu}(\overline{\Omega})$, for some $\mu>0$ depending on $(\gamma,[\partial\Omega]_{C^{0,1}})$. Specifically:
\begin{itemize}
    \item[(a)] If $\gamma>1$, and $\Omega$ is a $C^{1}$ or convex domain, then $u\in C^{\frac{2}{1+\gamma}}(\overline{\Omega})$ and
    \begin{equation*}
        u(X)\sim d(X,\partial\Omega)^{\frac{2}{1+\gamma}}.
    \end{equation*}
    \item[(b1)] If $\gamma<1$, and $\partial\Omega$ is a $C^{1}$ graph near the origin, then
    \begin{equation}\label{eq. estimate b1 in cor 1.1}
        u(X)\gtrsim d^{\alpha}(X,\partial\Omega),\quad\mbox{for any }\alpha>1.
    \end{equation}
    \item[(b2)] If $\gamma<1$, and $\partial\Omega$ is a concave graph near the origin, then
    \begin{equation}\label{eq. estimate b2 in cor 1.1}
        u(X)\gtrsim d(X,\partial\Omega).
    \end{equation}
\end{itemize}
\end{corollary}

Corollary~\ref{cor. of thm 1.2} (a) extends the sub-critical case ($\gamma > 1$) results of Gui and Lin \cite{GuLi} to a broader class of Lipschitz domains, as captured in the estimate \eqref{BBS}  mentioned earlier.

While Theorem~\ref{thm. growth rate} (b) provides a general estimate for critical cases, the following refinement establishes optimal growth under the additional assumption \eqref{eq. growth rate special condition in critical case} below, which holds for many natural cone configurations. The estimate matches the lower bound from Theorem~\ref{thm. growth rate} (e), confirming its optimality.

\begin{theorem}[improved growth rate estimate]\label{thm. growth rate, (b) improved}
    Under assumptions of Theorem~\ref{thm. growth rate} (b) and further we assume that there exists a solution $w\in C(\overline{Cone_{\Sigma}\cap B_{1}})$ to the Dirichlet problem
    \begin{equation}\label{eq. growth rate special condition in critical case}
        \left\{\begin{aligned}
            &-\Delta w=H_{\Sigma}^{-\gamma}&\mbox{ in }&Cone_{\Sigma}\cap B_{1}\\
            &w=0&\mbox{ on }&\partial(Cone_{\Sigma}\cap B_{1})
        \end{aligned}\right.,
    \end{equation}
    where $H_{\Sigma}$ is defined in \eqref{eq. H Sigma, homogeneous harmonic in a cone}. Then we have
    \begin{equation*}
        u(X)\leq C R^{\frac{-2}{1+\gamma}}\|u\|_{L^{\infty}(\mathcal{GC}_{3R})}\cdot|\frac{X}{R}|^{\phi_{\Sigma}}(\ln{\frac{2R}{|X|}})^{\phi_{\Sigma}/2},\ \forall X\in\mathcal{GC}_{R}.
    \end{equation*}
\end{theorem}
\begin{remark}
Despite apparent similarities, Theorem~\ref{thm. growth rate} (e) and Theorem~\ref{thm. growth rate, (b) improved} do not directly imply \eqref{BBS} in the critical case ($\gamma=1$), unless the boundary is totally flat.
\end{remark}

We next establish a comparison result extending the classical boundary Harnack principle to the singular semi-linear setting. Historically speaking, the boundary Harnack principle (to be stated in Lemma~\ref{lem. classical Kemper}) was initially studied in the seminal work of Kemper \cite{K72} for the Laplace equation, and was subsequently extended to general linear elliptic equations in \cite{AS19,BB91,BB94,CFMS81,FGMS88,JK82,K83,RT21}.
\begin{theorem}[comparison between solutions]\label{thm. Harnack Comparable, general Lipschitz domain}
    Assume that $[g(x')]_{C^{0,1}}\leq L$ and $g(0)=0$. If $u,v\geq0$ both satisfy \eqref{eq. main} in $\mathcal{GC}_{3R}$ with $0<\lambda\leq f(X)\leq\Lambda$, then there exists a constant $C=C(n,L,\gamma,\lambda,\Lambda)$ such that
    \begin{equation}\label{eq. thm 1.4 (1)}
        C^{-1}\min\Big\{\frac{\|u\|_{L^{\infty}(\mathcal{GC}_{3R})}}{\|v\|_{L^{\infty}(\mathcal{GC}_{3R})}},1\Big\}\leq\frac{u}{v}\leq C\max\Big\{\frac{\|u\|_{L^{\infty}(\mathcal{GC}_{3R})}}{\|v\|_{L^{\infty}(\mathcal{GC}_{3R})}},1\Big\}\mbox{ in }\mathcal{GC}_{R}.
    \end{equation}
    Besides, we have
    \begin{equation}\label{eq. thm 1.4 (2)}
        C^{-1}\min\Big\{\frac{u(R\vec{e_{n}})}{v(R\vec{e_{n}})},1\Big\}\leq\frac{u}{v}\leq C\max\Big\{\frac{u(R\vec{e_{n}})}{v(R\vec{e_{n}})},1\Big\}\mbox{ in }\mathcal{GC}_{R},
    \end{equation}
    and
    \begin{equation}\label{eq. thm 1.4 (3)}
        C^{-1}\frac{R^{\frac{2}{1+\gamma}}}{\|v\|_{L^{\infty}(\mathcal{GC}_{3R})}}\leq\frac{u}{v}\leq C\frac{\|u\|_{L^{\infty}(\mathcal{GC}_{3R})}}{R^{\frac{2}{1+\gamma}}}\mbox{ in }\mathcal{GC}_{R}.
    \end{equation}
\end{theorem}

While Theorem~\ref{thm. Harnack Comparable, general Lipschitz domain} establishes the boundedness of the ratio $\frac{u}{v}$, it does not imply the boundary regularity property of $\frac{u}{v}$. This exhibits a fundamental difference from the classical boundary Harnack principle for harmonic functions. The non-continuity result below is to be shown by analyzing an example \eqref{eq. strange domain} in Section~\ref{sec. examples}.
\begin{theorem}[non-continuity of the ratio]\label{thm. ratio not continuous}
    There exists an example $(g,\gamma,u,v)$ satisfying all assumptions in Theorem~\ref{thm. Harnack Comparable, general Lipschitz domain} (with $f(X)\equiv1$), such that the ratio $\ds\frac{u}{v}$ fails to be continuous at the boundary $\Gamma$.
\end{theorem}

Despite the non-continuity example in Theorem~\ref{thm. ratio not continuous}, we are still able to prove the continuity of the ratio $u/v$ under some additional assumptions on the geometry of the boundary. For example, based on Definition~\ref{def. cone classification}, we establish the continuity of the ratio $\frac{u}{v}$ at the boundary in the sub-critical and the critical case by showing that its limit at the boundary must be $1$. Notice that the lower semi-continuity in Theorem~\ref{thm. limit is 1} becomes full continuity by the symmetry in $\frac{u}{v}$ and $\frac{v}{u}$.
\begin{theorem}[continuity of the ratio, case 1]\label{thm. limit is 1}
    Under the assumptions in Theorem~\ref{thm. Harnack Comparable, general Lipschitz domain} and further assume that $\mathcal{GC}_{3R}\subseteq\overline{Cone_{\Sigma}}$ for an open spherical domain $\Sigma$ with Lipschitz boundary. Let $Q\leq1$ such that
    \begin{equation*}
        \frac{u}{v}-1\geq-Q\mbox{ in }\mathcal{GC}_{3R}.
    \end{equation*}
    We can then find some
    \begin{equation*}
        \epsilon\geq c(n,L,\gamma,\lambda,\Lambda,\phi_\Sigma)\frac{R^{2}}{\|v\|_{L^{\infty}(\mathcal{GC}_{3R})}^{1+\gamma}}>0,
    \end{equation*}
    such that:
    \begin{itemize}
        \item[(a)] If the cone $Cone_{\Sigma}$ is sub-critical, then
        \begin{equation*}
            \frac{u}{v}-1\geq-Q|\frac{X}{R}|^{\epsilon}\quad\mbox{for }X\in\mathcal{GC}_{R};
        \end{equation*}
        \item[(b)] If the cone $Cone_{\Sigma}$ is critical and \eqref{eq. growth rate special condition in critical case} has a continuous solution, then
        \begin{equation*}
            \frac{u}{v}-1\geq-Q(\ln{\frac{2R}{|X|}})^{-\epsilon}\quad\mbox{for }X\in\mathcal{GC}_{R}.
        \end{equation*}
    \end{itemize}
\end{theorem}

Finally, in the super-critical case, we sometimes still have the continuity of $\frac{u}{v}$, given that the boundary $\Gamma$ is $C^{1}$ or convex. Unlike Theorem~\ref{thm. limit is 1}, the limiting ratio need not equal $1$.
\begin{theorem}[continuity of the ratio, case 2]\label{thm. limit is continuous when obtuse}
    Under the assumptions in Theorem~\ref{thm. Harnack Comparable, general Lipschitz domain}, and we further assume that
    \begin{itemize}
        \item Either: $\Gamma$ is a $C^{1}$ graph near the origin and $\gamma\in(0,1)\cup(1,\infty)$;
        \item Or: $\Gamma$ is the graph of a convex function $g$ near the origin and $\gamma>0$.
    \end{itemize}
    Then the ratio $\ds\frac{u}{v}$ is continuous at the boundary near the origin.
\end{theorem}
\begin{remark}
    In the first case of Theorem~\ref{thm. limit is continuous when obtuse}, the situation $\gamma=1$ is not included because it is beyond the methods used in this paper. The main difficulty is that a $C^{1}$ graph might not lie entirely below or above the tangent plane (which is a critical cone when $\gamma=1$) at a boundary point.
\end{remark}
\begin{remark}
    Given Theorem~\ref{thm. limit is continuous when obtuse}, we can expect that the boundary curve \eqref{eq. strange domain} to be used in proving Theorem~\ref{thm. ratio not continuous} is not a convex graph.
\end{remark}

The idea in proving Theorem~\ref{thm. limit is continuous when obtuse} is that $u$ and $v$ are almost harmonic near the boundary. This is intuitively because $u$ and $v$ are so "large" near the boundary that the singular right-hand side has little contribution. With the help of the boundary Harnack principle \cite{K72}, we can approximate $u$ and $v$ each with a harmonic function, then $u/v$ is continuous near the boundary.

Precisely, we can fix a harmonic function $H\geq0$ defined in $B_{2}\cap\Omega$ such that it vanishes at $\Gamma$. By "the super-critical interior cone condition" \eqref{eq. super-critical interior cone condition}, we are able to, without loss of generality, assume that for some $\gamma'<\gamma$,
\begin{equation}\label{eq. H is larger than 2/(1+gamma')}
    \|H\|_{L^{\infty}(B_{r}\cap\Omega)}\geq H(\frac{9}{10}r\vec{e_{n}})\geq r^{\frac{2}{1+\gamma'}},\quad\mbox{for all }r\leq1.
\end{equation}
In fact, the choice of $\gamma'$ is presented as follows: Since $Cone_{\Sigma}$ is super-critical (with respect to $\gamma$), i.e.: $\frac{2}{1+\gamma}>\phi_{\Sigma}$, there exists some $\gamma'<\gamma$, such that
\begin{equation}\label{eq. range of gamma'}
    \frac{2}{1+\gamma}>\frac{2}{1+\gamma'}>\phi_{\Sigma}.
\end{equation}

The key is to show the following pointwise Schauder-type estimate using the Campanato iteration when "the interior condition" \eqref{eq. super-critical interior cone condition} holds. It indicates that $u$ can be approximated by a harmonic function.
\begin{theorem}[pointwise Schauder estimate]\label{thm. Schauder-Campanato}
    Under the assumptions of Theorem~\ref{thm. Harnack Comparable, general Lipschitz domain} with $\Gamma$ being $C^{1}$ or convex, and assume that we have "the super-critical interior cone condition" \eqref{eq. super-critical interior cone condition} with $Cone_{\Sigma}$ in $B_{2}\cap\Omega$. Let's also fix a harmonic function $H\geq0$ in $B_{2}\cap\Omega$ such that it vanishes at $\Gamma$ and satisfies \eqref{eq. H is larger than 2/(1+gamma')}. Then there exist positive constants $C,\epsilon>0$ with
    \begin{equation*}
        C=C(n,L,\gamma,Cone_{\Sigma},\lambda,\Lambda,\|u\|_{L^{\infty}(B_{1}\cap\Omega)}),\quad\epsilon=\epsilon(n,L,\gamma,Cone_{\Sigma}),
    \end{equation*}
    and a harmonic function
    \begin{equation*}
        h(X)=\mathcal{A}\cdot H(X)\mbox{ for some constant }C^{-1}\leq\mathcal{A}\leq C,
    \end{equation*}
    such that for $X\in B_{1/2}\cap\Omega$,
    \begin{equation*}
        h(X)\geq C^{-1}H(X)\geq C^{-1}|x_{n}-g(x')|^{\frac{2}{1+\gamma'}},
    \end{equation*}
    and
    \begin{equation*}
        |u(X)-h(X)|\leq C|X|^{\epsilon}\Big(|X|^{\frac{2}{1+\gamma'}}+H(X)\Big).
    \end{equation*}
\end{theorem}
\begin{remark}
    In fact, Theorem~\ref{thm. Schauder-Campanato} directly implies that
    \begin{equation*}
        \frac{u(X)}{H(X)}=\mathcal{A}+O(|X|^{\epsilon})\quad\mbox{near the origin}.
    \end{equation*}
    When in particular $\gamma<1$ and $\partial\Omega\in C^{1,DMO}$, then the estimate \eqref{BBS} follows immediately from the Hopf-Oleinik boundary point lemma (see \cite{DJV24,RSS23}).
\end{remark}
The reason why Theorem~\ref{thm. Schauder-Campanato} is called "Schauder" is partly because it becomes the $C^{1,\epsilon}$ boundary estimate if $\Gamma$ is smooth. But more fundamentally, the boundary Harnack principle for linear equations (like \cite{K72}) is to Theorem~\ref{thm. Schauder-Campanato}, as the Green function is to the standard $C^{2,\alpha}$ Schauder estimate. Recently, we have learned that in \cite{DS25}, De Silva and Savin have used a similar method to study the boundary Harnack principle for a more general class of linear equations in the form $tr(A(x)D^{2}u) + b(x)\cdot\nabla u+c(x)u=0$.

\subsection{Organization of the paper}
In Section~\ref{sec. preliminaries} we discuss the well-posedness of the "SLEF" and prove Theorem~\ref{thm. well-posedness}. In Section~\ref{sec. comparison between solutions} we prove Theorem~\ref{thm. Harnack Comparable, general Lipschitz domain}, while leaving the proof of Theorem~\ref{thm. growth rate} and Theorem~\ref{thm. growth rate, (b) improved} in Section~\ref{sec. growth rate estimate}. In Section~\ref{sec. continuity of the ratio u/v} we prove Theorem~\ref{thm. limit is 1}, Theorem~\ref{thm. limit is continuous when obtuse}, and Theorem~\ref{thm. Schauder-Campanato}. In Section~\ref{sec. examples}, we provide some examples to help the readers understand the main theorems better, and we construct one more example in order to prove Theorem~\ref{thm. ratio not continuous}.

\section{Preliminaries}\label{sec. preliminaries}
\subsection{Basic maximal principles}
We first present the maximal principle  for the following problem 
\begin{equation}\label{eq. main1}
    \left\{\begin{aligned}
        &-\Delta u=f(X)\cdot u^{-\gamma}&\mbox{ in }&\Omega\\
        &u=\varphi&\mbox{ on }&\partial\Omega
    \end{aligned}
    \right.,
\end{equation}
 where $\Omega$ is a bounded Lipschitz domain, $0<\lambda\leq f(X)\leq\Lambda$ in $\Omega$.
\begin{lemma}[maximal principle]\label{lem. maximal principle}
Assume that $u$ and $v$ are classical super-solution and sub-solution 
 of \eqref{eq. main1}, respectively, that is, $u,v\in C^2(\Omega)\cap C(\overline{\Omega})$ satisfy:
 \begin{equation}\label{eq. main supersolution}
    \left\{\begin{aligned}
        &-\Delta u\geq f(X)\cdot 
        u^{-\gamma}&\mbox{ in }&\Omega\\
        &u\geq\varphi&\mbox{ on }&\partial\Omega
    \end{aligned}
    \right.,
\end{equation}
and
\begin{equation}\label{eq. main subsolution}
    \left\{\begin{aligned}
        &-\Delta v\leq f(X)\cdot v^{-\gamma}&\mbox{ in }&\Omega\\
        &v\leq\varphi&\mbox{ on }&\partial\Omega
    \end{aligned}
    \right..
\end{equation}
Then 
\[u\geq v \ \mbox{in}\ \Omega.\]
\end{lemma}
\begin{proof}
 We proceed by contradiction. Suppose the conclusion does not hold, then there exists a point  $X_0\in\Omega$ such that $$(u-v)(X_0)=\min\limits_{X\in \Omega}(u-v)(X)<0.$$
  Consequently, from the second derivative test we have \begin{equation}\label{cnt} -\Delta (u-v)(X_0)\leq 0.\end{equation}
  In addition, by \eqref{eq. main supersolution} and \eqref{eq. main subsolution}, we obtain
  \begin{equation*}
    -\Delta (u-v)(X_0)\geq f(X)\cdot (u(X_0)^{-\gamma}-v(X_0)^{-\gamma})>0,
\end{equation*}
this contradicts \eqref{cnt}. 

Therefore, we conclude that $u\geq v\ \ \mbox{in}\  \Omega,$ and this completes the proof.
\end{proof}

The following lemma implies that a solution to the "SLEF" is non-degenerate.
\begin{lemma}[interior lower bound]\label{lem. non degenerate}
    Assume that $u\geq0$ satisfies
    \begin{equation*}
        -\Delta u\geq\lambda\cdot u^{-\gamma}\mbox{ in }B_{r},
    \end{equation*}
    then
    \begin{equation*}
        u(0)\geq c(n,\gamma,\lambda)r^{\frac{2}{1+\gamma}}.
    \end{equation*}
\end{lemma}
\begin{proof}
    Without loss of generality, we assume $r=1$ using the invariant rescale
    \begin{equation*}
        \widetilde{u}(Y)=r^{-\frac{2}{1+\gamma}}u(rY),\quad Y\in B_{1}.
    \end{equation*}
    To show $u(0)\geq0$, we construct a barrier function
    \begin{equation*}
        v=\Big((\frac{2n}{\lambda})^{-1/\gamma}-|X|^{2}\Big)_{+}.
    \end{equation*}
    We see that
    \begin{equation*}
        -\Delta v\leq2n\leq \lambda v^{-\gamma}\mbox{ in }B_{1},
    \end{equation*}
    and $v=0\leq u$ on $\partial B_{1}$. Therefore,
    \begin{equation*}
        u(0)\geq(\frac{2n}{\lambda})^{-1/\gamma}=:c(n,\gamma,\lambda)
    \end{equation*}
    as the equation $-\Delta u=\lambda\cdot u^{-\gamma}$ has the maximal principle, as shown in Lemma~\ref{lem. maximal principle}.
\end{proof}
Its direct corollary is the following interior Harnack principle.
\begin{corollary}[interior Harnack principle]\label{cor. interior Harnack}
    Assume that $u\geq0$ satisfies
    \begin{equation*}
        -\Delta u=f(X)\cdot u^{-\gamma}\mbox{ in }B_{r}
    \end{equation*}
    with $f(X)\in[\lambda,\Lambda]$, then
    \begin{equation*}
        \max_{B_{r/2}}u\leq C(n,\gamma,\lambda,\Lambda)\min_{B_{r/2}}u.
    \end{equation*}
\end{corollary}
\begin{proof}
    We first apply Lemma~\ref{lem. non degenerate} and obtain a lower bound of $u$ in $B_{0.9r}$. Such a lower bound also implies an upper bound of $f(X)u^{-\gamma}$. By the standard Harnack principle,
    \begin{equation*}
        \max_{B_{r/2}}u\leq C\Big(\min_{B_{r/2}}u+r^{2}\Lambda(\min_{B_{0.9r}}u)^{-\gamma}\Big)\leq C\Big(\min_{B_{r/2}}u+r^{\frac{2}{1+\gamma}}\Big).
    \end{equation*}
    Moreover, the last term $r^{\frac{2}{1+\gamma}}$ can also be absorbed into $\ds\min_{B_{r/2}}u$ using Lemma~\ref{lem. non degenerate}.
\end{proof}
\subsection{Proof of Theorem~\ref{thm. well-posedness}}
The uniqueness of the problem \eqref{eq. main1} can be  obtained from Lemma \ref{lem. maximal principle}. Next, we provide the detailed proof of the existence of \eqref{eq. main1}.
  \begin{itemize}
        \item Step 1: We show that there exists a classical sub-solution $h_{\varphi}$ for \eqref{eq. main1}.
        
        Assume that $h_\varphi$ is a harmonic replacement of \eqref{eq. main1}, that is,
    \begin{equation*}
    \left\{\begin{aligned}
        &-\Delta h_{\varphi}=0&\mbox{ in }&\Omega\\
    &h_{\varphi}=\varphi&\mbox{ on }&\partial\Omega
    \end{aligned}
    \right..
\end{equation*}
 where   $\Omega$ is a bounded  Lipschitz domain.   
 
 Then, $h_\varphi\in C^2(\Omega)\cap C(\overline{\Omega})$ is a classical sub-solution of \eqref{eq. main1}.
  \medskip
  
\item Step 2: Under the condition \(\min\limits_{X \in \partial \Omega} \varphi(X) =: m > 0\), we show the existence of  a classical solution to problem \eqref{eq. main1} via the  the continuity method.

 \medskip
 
 Consider a family of parameterized problems 
\begin{equation}\label{ut}
    \left\{\begin{aligned}
        &-\Delta u_t(X)=tf(X)u_t^{-\gamma}(X)&\mbox{ in }&\Omega\\
    &u_t=\varphi&\mbox{ on }&\partial\Omega
    \end{aligned}
    \right.,
\end{equation}
with $t\in[0,1].$ Define \begin{equation*}
    S:=\{t\in[0,1], \mbox{ the problem }\eqref{ut}\mbox{ has a classical solution}\}.
\end{equation*}

Clearly, the harmonic function $h_{\varphi}$ is a classical solution of \eqref{ut} with $t=0$. Hence $0\in S$, ensuring $S\neq \emptyset$.

By Step 1 and that $u_{t}$ is super-harmonic,
\[u_t(X)\geq h_{\varphi}(X)\geq m>0, \ \forall\ X\in \overline{\Omega},\ \forall t\in [0,1].\]
Therefore,  
the nonlinear term $tf(X)u_t(X)^{-\gamma}$ in \eqref{ut} is nonnegative and uniformly bounded in $\Omega$.

Since $f\in C^{\alpha}_{loc}(\Omega)$, the classical elliptic regular theory implies that
\begin{equation*}
    u_t\in C_{loc}^{2,\alpha}(\Omega)\cap C(\overline{\Omega})\mbox{ for each }t\in[0,1].
\end{equation*}
Precisely, there exist universal constants $\epsilon$, $C_{1}$ and $C_{2}(\Omega')$ (for every $\overline{\Omega'}\subseteq\Omega$) independent of $t$ such that
\begin{equation*}
    \|u_t-h_{\varphi}\|_{C^{\epsilon}_{0}(\overline{\Omega})}\leq C_{1},\ \|u_t\|_{C^{2,\alpha}(\Omega')}\leq C_{2}(\Omega'),\ \forall t\in [0,1].
\end{equation*}
Then one can easily verify that $S$ is both open and closed. 

Thus, by  the  the continuity method, we derive that $S=[0,1].$ In particular, when $t=1$, the original problem  \eqref{eq. main1} admits a classical solution.
\medskip

\item Step 3: When \(\min\limits_{X \in \partial \Omega} \varphi(X) =: m = 0\), we establish the existence of  classical solution to problem \eqref{eq. main1}. The proof proceeds as follows:

Consider the regularized problem:
\begin{equation}\label{ue}
    \left\{\begin{aligned}
        &-\Delta u_{\varepsilon}(X)=f(X)u_{\varepsilon}^{-\gamma}(X)&\mbox{ in }&\Omega\\
    &u_{\varepsilon}=\varphi+{\varepsilon}&\mbox{ on }&\partial\Omega
    \end{aligned}
    \right.,
\end{equation}
here $\varepsilon>0.$ From Step 2, we conclude that for each $\varepsilon>0$, the problem \eqref{ue} admits a classical solution $u_{\varepsilon}\in C^{2,\alpha}_{loc}(\Omega)\cap C(\overline{\Omega}).$ The local $C^{2,\alpha}$ norm is independent of $\epsilon$ for each fixed interior ball, due to Lemma~\ref{lem. non degenerate}.

By applying the maximal principle (Lemma \ref{lem. maximal principle}),  the sequence $\{u_{\varepsilon}\}$ is monotone  decreasing as ${\varepsilon} \searrow 0$, and satisfies
 \[u_{\varepsilon}(X)\geq h_{\varphi}(X),\ \forall X\in \overline{\Omega}.\] where 
$h_{\varphi}$ is the harmonic function with boundary data ${\varphi}.$

Now let $\varepsilon \to 0$, the monotone convergence theorem implies that 
$u_{\varepsilon}$ converges pointwise everywhere to a function 
$u_0\in C^{2,\alpha}_{loc}(\Omega)$, which satisfies
\begin{equation*}
    -\Delta u_{0}(X)=f(X)u_{0}^{-\gamma}(X)\mbox{ in }\Omega.
\end{equation*}
On the one hand, it can be easily seen that
\begin{equation*}
    \liminf_{X\to Y}u_{0}(X)\geq\varphi(Y),\quad\mbox{for every }Y\in\partial\Omega.
\end{equation*}
On the other hand, for each fixed $Y\in\partial\Omega$ and $\epsilon>0$, by the continuity of $u_{\epsilon/2}(X)$, we then have the existence of $\delta>0$ such that
\begin{equation*}
    \Big|u_{\epsilon/2}(X)-\big(\varphi(Y)+\frac{\epsilon}{2}\big)\Big|\leq\frac{\epsilon}{2},\quad\mbox{if }|X-Y|\leq\delta.
\end{equation*}
Then, for $|X-Y|\leq\delta$, we have
\begin{equation*}
    u_{0}(X)-\varphi(Y)\leq u_{\epsilon/2}(X)-\varphi(Y)\leq\epsilon,
\end{equation*}
and thus have also verified the boundary condition:
\begin{equation*}
    u_{0}=\varphi\mbox{ on }\partial\Omega.
\end{equation*}
Thus we established the existence of a classical solution to problem \eqref{eq. main1}.
\end{itemize}
This completes the proof of Theorem \ref{thm. well-posedness}.

\section{Comparison between solutions}\label{sec. comparison between solutions}
In this section, our goal is to prove Theorem~\ref{thm. Harnack Comparable, general Lipschitz domain}. The proof is robust and can be applied to more general uniformly elliptic operators, such as $\mathrm{div}(A\nabla u)$ or $\mathrm{tr}(A\cdot D^{2}u)$.

The key lemma below provides a weaker criterion for a function to be positive. The method is inspired by De Silva and Savin \cite{DS20}.
\begin{lemma}\label{lem. ensure positive}
    Let $\Gamma$ be a Lipschitz graph such that $[g]_{C^{0,1}}\leq L$. There exist constants $(M,\delta)$ depending only on $(n,L,C_{1})$ such that if the following holds:
    \begin{itemize}
        \item[(a)] $\Delta w\leq\frac{C_{1}}{dist(X,\Gamma)^{2}}\max\{w,0\}$ and $w\geq-1$ in the interior of $\mathcal{GC}_{r}$;
        \item[(b)] $w\geq M$ in $\mathcal{SC}_{r}=\mathcal{SC}_{r,\delta}(0)$;
        \item[(c)] $w=0$ on $\Gamma$, in the continuous sense or the trace sense.
    \end{itemize}
    Then it is guaranteed that $w\geq0$ on the vertical segment $\mathcal{GC}_{r}\cap\{x'=0\}$.
\end{lemma}
\begin{proof}
    It suffices to prove inductively the following two statements:
    \begin{itemize}
        \item[($P_{i}$):] $w\geq Ma^{i}$ in $\mathcal{SC}_{2^{-i}r}$;
        \item[($Q_{i}$):] $w\geq-a^{i}$ in $\mathcal{GC}_{2^{-i}r}$.
    \end{itemize}
    The choice of $(M,a,\delta)$ will be given at the end of the proof. Automatically we have ($P_{0}$) and ($Q_{0}$). Now we assume that ($P_{k}$) and ($Q_{k}$) are true for some $k\geq0$.

    Let's first prove ($P_{k+1}$). Pick an arbitrary point $X\in\mathcal{SC}_{2^{-k-1}r}\setminus\mathcal{SC}_{2^{-k}r}$ and write
        \begin{equation*}
            X=(x',g(x')+t),\quad x'\in B_{2^{-k-1}r}',\quad t\in[2^{-k-1}\delta r,2^{-k}\delta r].
        \end{equation*}
        We insert $p_{1},\cdots,p_{N}$ between $p_{0}=(x',g(x')+4t)$ and $p_{N+1}=X$ such that:
        \begin{itemize}
            \item[(a)] $p_{j}=(x',g(x')+t_{j})$ for $0\leq j\leq N+1$;
            \item[(b)] $t_{j}\geq t_{j+1}$ and $\ds(1-\frac{3}{4\sqrt{L^{2}+1}})t_{j}\leq(1-\frac{1}{4\sqrt{L^{2}+1}})t_{j+1}$ for $0\leq j\leq N$;
            \item[(c)] $N$ has a uniform upper bound depending only on $L$.
        \end{itemize}
        It is worth mentioning that it is inferred from (b) that
        \begin{equation*}
            B_{\frac{t_{j+1}}{4\sqrt{L^{2}+1}}}(p_{j+1})\subseteq B_{\frac{3t_{j}}{4\sqrt{L^{2}+1}}}(p_{j})\subseteq B_{\frac{t_{j}}{\sqrt{L^{2}+1}}}(p_{j})\subseteq\mathcal{GC}_{2^{-k}r},
        \end{equation*}
        and that
        \begin{equation*}
            B_{\frac{t_{0}}{4\sqrt{L^{2}+1}}}(p_{0})\subseteq\mathcal{SC}_{2^{-k}r}.
        \end{equation*}
        We denote
        \begin{equation*}
            m_{j}=\min\{w(Y):Y\in B_{\frac{t_{j}}{4\sqrt{L^{2}+1}}}(p_{j})\},
        \end{equation*}
        then we first have $m_{0}\geq Ma^{k}$ by the induction hypothesis ($P_{k}$). Assume that $m_{j}$ is known for some $0\leq j\leq N$, then we construct a barrier
        \begin{equation*}
            \beta_{j}(Y)=\max\Big\{\frac{m_{j}+a^{k}}{4^{q}-1}\Big(\frac{t_{j}/\sqrt{L^{2}+1}}{|y-p_{j}|}\Big)^{q}-\frac{m_{j}+4^{q}a^{k}}{4^{q}-1},-a^{k}\Big\},\quad\mbox{for some }q=q(n,L,C_{1}).
        \end{equation*}
        It turns out that as long as $q$ is sufficiently large,
        \begin{equation*}
            \Delta\beta_{j}\geq\frac{C_{1}}{dist(X,\Gamma)^{2}}\max\{\beta_{j},0\}\quad\mbox{in the annulus }B_{\frac{t_{j}}{\sqrt{L^{2}+1}}}(p_{j})\setminus B_{\frac{t_{j}}{4\sqrt{L^{2}+1}}}(p_{j}).
        \end{equation*}
        Besides,
        \begin{equation*}
            \beta_{j}=m_{j}\mbox{ on }\partial B_{\frac{t_{j}}{4\sqrt{L^{2}+1}}}(p_{j}),\quad\beta_{j}=-a^{k}\mbox{ on }\partial B_{\frac{t_{j}}{\sqrt{L^{2}+1}}}(p_{j}).
        \end{equation*}
        By maximal principle, we see that
        \begin{equation*}
            m_{j+1}\geq\min\{\beta_{j}(Y):Y\in B_{\frac{3t_{j}}{4\sqrt{L^{2}+1}}}(p_{j})\}\geq\frac{(4/3)^{q}-1}{4^{q}-1}m_{j}-\frac{4^{q}-(4/3)^{q}}{4^{q}-1}a^{k},
        \end{equation*}
        or in other words,
        \begin{equation*}
            m_{j+1}+a^{k}\geq\frac{(4/3)^{q}-1}{4^{q}-1}(m_{j}+a^{k}).
        \end{equation*}
        We will require
        \begin{equation}\label{eq. requirement 1 on M and a}
            \frac{Ma+1}{M+1}\leq\Big(\frac{(4/3)^{q}-1}{4^{q}-1}\Big)^{N+1},
        \end{equation}
        so that we have $w(X)\geq m_{N+1}\geq Ma^{k+1}$.

        We next show ($Q_{k+1}$). We extend $w$ trivially below the graph $\Gamma$ and consider
        \begin{equation*}
            w^{-}(Y):=\max\{-w(Y),0\}.
        \end{equation*}
        As $\Delta w\leq0$ in the interior of $\mathcal{GC}_{2^{-k}r}$, we see $\Delta w^{-}\geq0$ in the doubled cylinder
        \begin{equation*}
            \mathcal{DC}_{2^{-k}r}=\{(y',y_{n}):\ |y'|\leq2^{-k}r,\ |y_{n}-g(y')|\leq2^{-k}r\}.
        \end{equation*}
        We apply the De Giorgi - Nash - Moser estimate (or the weak Harnack principle) to $w^{-}$ and obtain
        \begin{equation*}
            \max_{\mathcal{GC}_{2^{-k-1}}}w^{-}\leq\frac{C(n,L)}{(2^{-k}r)^{n/2}}\|w^{-}\|_{L^{2}(\mathcal{DC}_{2^{-k}r})}\leq\frac{C(n,L)}{(2^{-k}r)^{n/2}}\|w^{-}\|_{L^{\infty}(\mathcal{DC}_{2^{-k}r})}|supp(w^{-})|^{1/2}.
        \end{equation*}
        Notice that $\|w^{-}\|_{L^{\infty}(\mathcal{DC}_{2^{-k}r})}\leq a^{k}$ by the assumption ($Q_{k}$), and
        \begin{equation*}
            supp(w^{-})\subseteq\{(y',y_{n}):\ |y'|\leq2^{-k}r,\ g(y')\leq y_{n}\leq g(y')+2^{-k}\delta r\}.
        \end{equation*}
        Therefore, we require
        \begin{equation}\label{eq. requirement 2 on a and delta}
            a\geq C(n,L)\sqrt{\delta},
        \end{equation}
        and obtain ($Q_{k+1}$) by
        \begin{equation*}
            \max_{\mathcal{GC}_{2^{-k-1}}}w^{-}\leq C(n,L)a^{k}\sqrt{\delta}\leq a^{k+1}.
        \end{equation*}

        Finally, we need to choose $(M,a,\delta)$ satisfying \eqref{eq. requirement 1 on M and a} and \eqref{eq. requirement 2 on a and delta}. We can choose $M$ large and $\delta$ small, then there is room for us to choose $a$.
\end{proof}

The next lemma (also inspired by \cite{DS20}) shows that if $u$ satisfies \eqref{eq. main}, then $u$ "is more likely to" achieve its maximal value in the suspended cylinder.
\begin{lemma}\label{lem. u is small near the boundary}
    Assume that $u$ satisfies \eqref{eq. main} in $\mathcal{GC}_{2r}$. Then
    \begin{equation*}
        C^{-1}\max_{\mathcal{GC}_{r}}u\leq u(r\vec{e_{n}})\leq C\min_{\mathcal{SC}_{r,\delta}}u,
    \end{equation*}
    for some $C=C(n,L,\gamma,\lambda,\Lambda,\delta)$.
\end{lemma}
\begin{proof}
    For every $X=(x',x_{n})\in\mathcal{GC}_{\frac{3}{2}r}$, we can insert finitely many points between $X$ and $r\vec{e_{n}}$, so that adjacent points are closer to each other than to the boundary $\Gamma$. Moreover, the number of points to be inserted is
    \begin{equation*}
        N(X)=O\Big(\ln{r}-\ln{|x_{n}-g(x')|}\Big).
    \end{equation*}
    
    We apply the interior Harnack principle (Corollary~\ref{cor. interior Harnack}) between adjacent points (so Corollary~\ref{cor. interior Harnack} is applied $N(X)+1$ times), and obtain that
    \begin{equation}\label{eq. u(X) can not tend to infinity too fast}
        C^{-1}r^{-p}|x_{n}-g(x')|^{p}u(r\vec{e_{n}})\leq u(X)\leq Cr^{p}|x_{n}-g(x')|^{-p}u(r\vec{e_{n}}),\quad\mbox{for }X\in\mathcal{GC}_{\frac{3}{2}r}.
    \end{equation}
    Here, $C$ and $p$ depend on $(n,L,\gamma,\lambda,\Lambda)$. In particular,
    \begin{equation*}
        \min_{\mathcal{SC}_{r,\delta}}u\geq c(n,L,\gamma,\lambda,\Lambda,\delta)\cdot u(r\vec{e_{n}}).
    \end{equation*}
    Moreover, it follows from \eqref{eq. u(X) can not tend to infinity too fast} that for some small $\epsilon=\epsilon(n,L,\gamma,\lambda,\Lambda)$, the "pseudo $L^{\epsilon}$ norm"
    \begin{equation*}
        \|u(X)\|_{L^{\epsilon}(\mathcal{GC}_{\frac{3}{2}r})}:=\Big(\int_{\mathcal{GC}_{\frac{3}{2}r}}u(X)^{\epsilon}dX\Big)^{1/\epsilon}\leq C(n,L,\gamma,\lambda,\Lambda)r^{\frac{n}{\epsilon}}u(r\vec{e_{n}}).
    \end{equation*}
    We then apply the weak Harnack principle to the truncated function
    \begin{equation*}
        v(X)=\max\{u(X)-u(r\vec{e_{n}}),0\},
    \end{equation*}
    which satisfies
    \begin{equation*}
        -\Delta v\leq\Lambda\cdot u(r\vec{e_{n}})^{-\gamma}
    \end{equation*}
    in the distributional sense. Then
    \begin{equation*}
        \max_{\mathcal{GC}_{r}}u\leq u(r\vec{e_{n}})+Cr^{-\frac{n}{\epsilon}}\|u(X)\|_{L^{\epsilon}(\mathcal{GC}_{\frac{3}{2}r})}+Cr^{2}\Lambda\cdot u(r\vec{e_{n}})^{-\gamma}.
    \end{equation*}
    The first two terms are easily controlled by $u(r\vec{e_{n}})$. The third term can also be controlled by $u(r\vec{e_{n}})$ by the lower bound estimate Lemma~\ref{lem. non degenerate}.
\end{proof}

The final lemma is a algebraic fact:
\begin{lemma}\label{lem. algebraic fact}
    Let $\epsilon\in(0,1)$ be arbitrary, let $u,v>0$, and let $w=u-\epsilon v$. Then, there exists a constant $C(\gamma)$ independent of $(\epsilon,u,v)$, such that
    \begin{equation*}
        \epsilon v^{-\gamma}-u^{-\gamma}\leq\frac{C(\gamma)}{v^{1+\gamma}}\max\{w,0\}.
    \end{equation*}
\end{lemma}
\begin{proof}
    When $u\leq v$, then we clearly have $\epsilon v^{-\gamma}-u^{-\gamma}\leq0$. When $u\geq v$, then we consider the case $\epsilon\leq\frac{1}{2}$ and the case $\epsilon\geq\frac{1}{2}$ separately.
    
    If $\epsilon\leq\frac{1}{2}$, then $w\geq\frac{v}{2}$, while $\epsilon v^{-\gamma}-u^{-\gamma}\leq\frac{1}{2}v^{-\gamma}$, so we have $\epsilon v^{-\gamma}-u^{-\gamma}\leq\frac{w}{v^{1+\gamma}}$.
    
    If $\epsilon\geq\frac{1}{2}$, by noticing that $-u^{-\gamma}$ is convex in $u=\epsilon v+w$, we have $u^{-\gamma}\geq(\epsilon v)^{-\gamma}-\gamma(\epsilon v)^{-1-\gamma}w$ by linearization. As a result, $\epsilon v^{-\gamma}-u^{-\gamma}\leq\gamma(\epsilon v)^{-1-\gamma}w\leq\frac{\gamma\cdot2^{1+\gamma}}{v^{1+\gamma}}w$.
\end{proof}

Now we are ready to prove Theorem~\ref{thm. Harnack Comparable, general Lipschitz domain}. The strategy in proving Theorem~\ref{thm. Harnack Comparable, general Lipschitz domain} is to define
\begin{equation*}
    w=u-\epsilon v
\end{equation*}
for a small $\epsilon$. By showing $w\geq0$ in a small neighborhood of $0$, one can obtain a lower bound of $\ds\frac{u}{v}$.
\begin{proof}[Proof of Theorem~\ref{thm. Harnack Comparable, general Lipschitz domain}]
    We can easily obtain that
    \begin{equation*}
        \min_{\mathcal{GC}_{2R}}w\geq-\epsilon\|v\|_{L^{\infty}(\mathcal{GC}_{3R})}.
    \end{equation*}
    Moreover, by Lemma~\ref{lem. u is small near the boundary}, there exist some $c_{1}=c_{1}(n,L,\gamma,\lambda,\Lambda,\delta)$, such that
    \begin{equation*}
        u(X)\geq c_{1}\|u\|_{L^{\infty}(\mathcal{GC}_{3R})},\quad v(X)\geq c_{1}\|v\|_{L^{\infty}(\mathcal{GC}_{3R})},\quad\mbox{for every }X\in\mathcal{SC}_{2R}.
    \end{equation*}
    Let's then choose a small $\epsilon$ such that
    \begin{equation}\label{eq. 1.4 choice of epsilon}
        \epsilon=\frac{c_{1}\min\{\|u\|_{L^{\infty}(\mathcal{GC}_{3R})},\|v\|_{L^{\infty}(\mathcal{GC}_{3R})}\}}{(M+1)\|v\|_{L^{\infty}(\mathcal{GC}_{3R})}}\geq\frac{c_{2}R^{\frac{2}{1+\gamma}}}{(M+1)\|v\|_{L^{\infty}(\mathcal{GC}_{3R})}}
    \end{equation}
    where the last inequality follows from Lemma~\ref{lem. non degenerate}, then $\ds\min_{\mathcal{SC}_{2R}}w\geq M\epsilon\|v\|_{L^{\infty}(\mathcal{GC}_{3R})}$. Besides, by Lemma~\ref{lem. non degenerate} and Lemma~\ref{lem. algebraic fact}, we notice that
    \begin{equation*}
        \Delta w\leq\frac{C}{dist(X,\Gamma)^{2}}\max\{w,0\}.
    \end{equation*}
    Therefore, we can apply Lemma~\ref{lem. ensure positive} for every $X=(x',g(x'))$ with $|x'|\leq R$ and obtain that
    \begin{equation*}
        w(Y)\geq0\mbox{ in }\mathcal{GC}_{R}(X)\cap\{y'=x'\},
    \end{equation*}
    which implies that $u\geq\epsilon v$ in $\mathcal{GC}_{R}$. From the choice of $\epsilon$ in \eqref{eq. 1.4 choice of epsilon}, we obtain \eqref{eq. thm 1.4 (1)} and \eqref{eq. thm 1.4 (3)}. Finally, in order to obtain \eqref{eq. thm 1.4 (2)}, we use Lemma~\ref{lem. u is small near the boundary} to get
    \begin{equation*}
        C^{-1}\frac{\|u\|_{L^{\infty}(\mathcal{GC}_{3R})}}{\|v\|_{L^{\infty}(\mathcal{GC}_{3R})}}\leq\frac{u(R\vec{e_{n}})}{v(R\vec{e_{n}})}\leq C\frac{\|u\|_{L^{\infty}(\mathcal{GC}_{3R})}}{\|v\|_{L^{\infty}(\mathcal{GC}_{3R})}},
    \end{equation*}
    then \eqref{eq. thm 1.4 (2)} follows immediately from \eqref{eq. thm 1.4 (1)}.
\end{proof}

It is now the time that we state the classical boundary Harnack principle. Here, we only state the original version obtained by Kemper \cite{K72} in year 1972, despite the fact that the theory has been widely explored afterwards.
\begin{lemma}[classical boundary Harnack principle \cite{K72}]\label{lem. classical Kemper}
    Assume that $g(x'):\mathbb{R}^{n-1}\to\mathbb{R}$ is a Lipschitz function near the origin, with $g(0)=0$ and $\|g\|_{C^{0,1}(B_{1}')}\leq L$. Assume that $u,v\geq0$ are harmonic in $\mathcal{GC}_{1}$, and that $u,v$ both vanish on the curve $\{x_{n}=g(x')\}$ in the trace or limit sense. Then,
    \begin{equation*}
        C^{-1}\frac{u(\frac{1}{2}\vec{e_{n}})}{v(\frac{1}{2}\vec{e_{n}})}\leq\frac{u(x)}{v(x)}\leq C\frac{u(\frac{1}{2}\vec{e_{n}})}{v(\frac{1}{2}\vec{e_{n}})}\mbox{ in }\mathcal{GC}_{1/2}
    \end{equation*}
    for some $C=C(n,L)$. Moreover, there exists some $\epsilon=\epsilon(n,L)>0$, such that
    \begin{equation*}
        \Big\|\frac{u}{v}\Big\|_{C^{\epsilon}(\mathcal{GC}_{1/2})}\leq C\frac{u(\frac{1}{2}\vec{e_{n}})}{v(\frac{1}{2}\vec{e_{n}})}.
    \end{equation*}
\end{lemma}
The classical boundary Harnack principle will be made full use in the upcoming sections.

\section{Growth rate estimates}\label{sec. growth rate estimate}
In this section, our goal is to prove Theorem~\ref{thm. growth rate} and Theorem~\ref{thm. growth rate, (b) improved}.
\subsection{Reduction to solutions in a cone}
Let $L=[g]_{C^{0,1}}$, and let $u$ be a solution to \eqref{eq. main}. We consider the following rescaled function
\begin{equation*}
    u_{\lambda}(X)=\lambda^{-\frac{2}{1+\gamma}}u(\lambda X),\quad X\in\mathcal{GC}_{3\widetilde{R}},\quad\mbox{where }\lambda=\frac{R}{\sqrt{L^{2}+1}}\mbox{ and }\widetilde{R}=\sqrt{L^{2}+1}.
\end{equation*}
Through such a scaling, without loss of generality, we may assume that $R=\widetilde{R}=\sqrt{L^{2}+1}$ in Theorem~\ref{thm. growth rate} and Theorem~\ref{thm. growth rate, (b) improved}.

For a spherical region $\Sigma$ satisfying $\mathcal{GC}_{3\sqrt{L^{2}+1}}\subseteq\overline{Cone_{\Sigma}}$, we need to remove some unnecessary parts of $\Sigma$ to ensure that it is star-shaped, whose definition is given right below.
\begin{definition}
    Let $\Sigma\subseteq\partial B_{1}$ be an open spherical domain with a Lipschitz boundary. We then say $\Sigma$ is star-shaped, if the following holds:
    \begin{itemize}
        \item[(a)] $\vec{e_{n}}\in\Sigma$ while $-\vec{e_{n}}\notin\Sigma$;
        \item[(b)] For any point $\vec{p}\in\Sigma$ and $\vec{q}\in\partial B_{1}$, such that $\vec{q}$ lies on the shortest geodesic segment joining $\vec{p}$ and $\vec{e_{n}}$, we always have $\vec{q}\in\Sigma$.
    \end{itemize}
\end{definition}
Let's now reduce a region $\Sigma\subseteq\partial B_{1}$ to a star-shaped subset.
\begin{lemma}[reduction of a cone]\label{lem. star shape cone}
    Assume that the region $\mathcal{GC}_{3R}$ is contained in a cone $\overline{Cone_{\Sigma}}$. Then there exists a star-shaped region $\widetilde{\Sigma}\subseteq\Sigma$, such that we still have $\mathcal{GC}_{3R}\subseteq\overline{Cone_{\widetilde{\Sigma}}}$. Notice that $\widetilde{\Sigma}\subseteq\Sigma$ implies $\phi_{\widetilde{\Sigma}}\geq\phi_{\Sigma}$, so $Cone_{\widetilde{\Sigma}}$ is "less critical" than $Cone_{\Sigma}$.
\end{lemma}
\begin{proof}
    We naturally define $\widetilde{\Sigma}$ as the projection of $int(\mathcal{GC}_{3R})$ to $\partial B_{1}$. Next it suffices to verify several properties of $\widetilde{\Sigma}$.
    \begin{itemize}
        \item[(1)] $\widetilde{\Sigma}$ is open in the topology of $\partial B_{1}$: Since the projection map from $\mathbb{R}^{n}\setminus\{0\}\to\partial B_{1}$ is an open map, $\widetilde{\Sigma}$ as the image of $int(\mathcal{GC}_{3R})$ must be open.
        \item[(2)] $\widetilde{\Sigma}\subseteq\Sigma$: Because $\mathcal{GC}_{3R}\subseteq\overline{Cone_{\Sigma}}$.
        \item[(3)] $\mathcal{GC}_{3R}\subseteq\overline{Cone_{\widetilde{\Sigma}}}$: By taking limit of the inclusion $int(\mathcal{GC}_{3R})\subseteq Cone_{\widetilde{\Sigma}}$.
        \item[(4)] $\vec{e_{n}}\in\widetilde{\Sigma}$: Because $\vec{e_{n}}$ is the projection of $\epsilon\vec{e_{n}}\in int(\mathcal{GC}_{3R})$.
        \item[(5)] $-\vec{e_{n}}\notin\widetilde{\Sigma}$: Since $\Gamma$ is a graph, the ray $-t\vec{e_{n}}(t>0)$ has no intersection with $Cone_{\Sigma}$ (at least locally), so $-\vec{e_{n}}\notin\Sigma\supseteq\widetilde{\Sigma}$.
        \item[(6)] Convexity along rays: Assume that $\vec{p}\in\widetilde{\Sigma}$, so that there is a point
        \begin{equation*}
            X=(x',g(x')+t)\in int(\mathcal{GC}_{3R})
        \end{equation*}
        satisfying
        \begin{equation*}
            (x',g(x')+t)\parallel\vec{p}.
        \end{equation*}
        We construct a path joining $X$ and $t\vec{e_{n}}$, which is
        \begin{equation*}
            c(s)=(s x',g(s x')+t),\quad s\in[0,1].
        \end{equation*}
        It then follows that $c(s)$ is a continuous non-zero family of vectors spanned by $\vec{p}$ and $\vec{e_{n}}$. Moreover we have $c(s)$ is not in the negative direction of $\vec{e_{n}}$. We can then apply the intermediate value theorem, and obtain that for any $\vec{q}$ lying on the geodesic segment $[\vec{e_{n}},\vec{p}]$, there exists $s\in[0,1]$ such that $c(s)$ is in the positive direction of $\vec{q}$.
        \item[(7)] Regularity of $\widetilde{\Sigma}$: Since a neighborhood of $-\vec{e_{n}}$ is not included in $\widetilde{\Sigma}$, we only need to show $Cone_{\widetilde{\Sigma}}$ is a Lipschitz graph, and it immediately implies that $\partial\widetilde{\Sigma}$ is a Lipschitz graph with respect to azimuthal angle. First, from (6) we see $\partial Cone_{\widetilde{\Sigma}}$ is a graph. Next, let $\vec{p}\in Cone_{\widetilde{\Sigma}}$ be a unit vector, meaning that
        \begin{equation*}
            \lambda\vec{p}\in int(\mathcal{GC}_{3R}),\quad\mbox{for some }\lambda>0.
        \end{equation*}
        Since $\Gamma$ is a Lipschitz graph with $[g]_{C^{0,1}}\leq L$, there exists $\epsilon>0$ such that
        \begin{equation*}
            \{\vec{q}=(q',q_{n}):L|q'-\lambda p'|\leq q_{n}-\lambda p_{n}\leq\epsilon\}\subseteq int(\mathcal{GC}_{3R})\subseteq\overline{Cone_{\widetilde{\Sigma}}}.
        \end{equation*}
        After a scaling, we have
        \begin{equation*}
            \{\vec{q}=(q',q_{n}):L|q'-p'|\leq q_{n}-\lambda p_{n}\leq\frac{\epsilon}{\lambda}\}\subseteq\overline{Cone_{\widetilde{\Sigma}}},
        \end{equation*}
        so $\partial Cone_{\widetilde{\Sigma}}$ is a Lipschitz graph.
    \end{itemize}
    Therefore, we see $\widetilde{\Sigma}$ is what we need.
\end{proof}
By replacing $\Sigma$ with $\widetilde{\Sigma}$ (still denoted as $\Sigma$), we have that $Cone_{\Sigma}$ is a Lipschitz epigraph and that $\mathcal{GC}_{3\sqrt{L^{2}+1}}\subseteq\overline{Cone_{\Sigma}}$.

Since $f(X)\in[\lambda,\Lambda]$ is H\"older continuous in $\mathcal{GC}_{3\sqrt{L^{2}+1}}$, one can extend it to $\overline{Cone_{\Sigma}}$. The extended function, still denoted as $f(X)$, is H\"older continuous in $\overline{Cone_{\Sigma}}$ and satisfies $\lambda\leq f(X)\leq\Lambda$. Given a solution $u$ satisfying \eqref{eq. main} with $R=\sqrt{L^{2}+1}$, by noticing that
\begin{equation*}
    B_{2}\cap\Big(\partial\mathcal{GC}_{3\sqrt{L^{2}+1}}\setminus\Gamma\Big)=\emptyset,
\end{equation*}
one can verify that $u\cdot\chi_{\mathcal{GC}_{3\sqrt{L^{2}+1}}}$ is continuous on $\partial B_{2}$. Therefore, it follows from Theorem~\ref{thm. well-posedness} that there exists a unique solution $v$ to the following problem:
\begin{equation}\label{eq. v general solution in a cone}
    \left\{\begin{aligned}
        &-\Delta v=f(X)v^{-\gamma}&\mbox{in }&B_{2}\cap Cone_{\Sigma}\\
        &v=0&\mbox{on }&B_{2}\cap\partial Cone_{\Sigma}\\
        &v=u\cdot\chi_{\mathcal{GC}_{3\sqrt{L^{2}+1}}}&\mbox{on }&Cone_{\Sigma}\cap\partial B_{2}
    \end{aligned}\right..
\end{equation}
By Lemma~\ref{lem. maximal principle}, we have:
\begin{equation}\label{eq. v pointwise vs u pointwise}
    v\geq u\cdot\chi_{\mathcal{GC}_{3\sqrt{L^{2}+1}}}\mbox{ in }B_{2}\cap Cone_{\Sigma}.
\end{equation}
On the other hand, write $h=\|u\|_{L^{\infty}(\mathcal{GC}_{3\sqrt{L^{2}+1}})}$, then we have that $h\geq c(n,L,\gamma)$ by Lemma~\ref{lem. non degenerate}. By noticing that $v\leq h$ on $\partial(B_{2}\cap Cone_{\Sigma})$, we have $v\leq\overline{v}$ in $B_{2}\cap Cone_{\Sigma}$ for $\overline{v}(X)$ being the solution to the following problem:
\begin{equation*}
    \left\{\begin{aligned}
        &-\Delta\overline{v}=f(X)\overline{v}^{-\gamma}&\mbox{in }&B_{2}\cap Cone_{\Sigma}\\
        &\overline{v}=h&\mbox{on }&\partial(B_{2}\cap Cone_{\Sigma})
    \end{aligned}\right..
\end{equation*}
By noticing that $\overline{v}^{-\gamma}\leq h^{-\gamma}$, we have
\begin{equation*}
    v\leq\overline{v}\leq h+Ch^{-\gamma}\leq Ch\mbox{ in }B_{2}\cap Cone_{\Sigma},
\end{equation*}
hence
\begin{equation}\label{eq. v L infty vs u L infty}
    \|v\|_{L^{\infty}(B_{2}\cap Cone_{\Sigma})}\leq C\|u\|_{L^{\infty}(\mathcal{GC}_{3\sqrt{L^{2}+1}})}.
\end{equation}

By \eqref{eq. v pointwise vs u pointwise} and \eqref{eq. v L infty vs u L infty}, in Theorem~\ref{thm. growth rate} (a)-(c) and in Theorem~\ref{thm. growth rate, (b) improved}, it suffices to establish suitable upper bound estimates for $v$.

The discussion above suggests that one can reduce the study of solutions in an arbitrary Lipschitz domain to that in Lipschitz cones. For simplicity, when proving Theorem~\ref{thm. growth rate} (a)-(c) and Theorem~\ref{thm. growth rate, (b) improved}, we assume that $Cone_{\Sigma}$ and $\mathcal{GC}_{3\sqrt{L^{2}+1}}$ are identical in $B_{2}$. Consequently, the solution $u$ to \eqref{eq. main} (which is now identical to $v$) is well-defined in $B_{2}\cap Cone_{\Sigma}$, with $\Sigma$ being a star-shaped open spherical domain with a Lipschitz boundary.

\subsection{Proof of Theorem~\ref{thm. growth rate}}
From the discussion in the previous subsection, when proving Theorem~\ref{thm. growth rate} (a)-(c), $\mathcal{GC}_{3R}$ (with $R=\sqrt{L^{2}+1}$) and $Cone_{\Sigma}$ are assumed to be identical in $B_{2}$ and that $\Sigma$ is star-shaped. We begin with a few preliminary steps and simplifications.

By Theorem~\ref{thm. Harnack Comparable, general Lipschitz domain}, we can choose a specific function $U$ defined in $B_{2}\cap Cone_{\Sigma}$. By studying its growth rate at the origin, the growth rate of general solutions follows from Theorem~\ref{thm. Harnack Comparable, general Lipschitz domain}. Our choice of $U$ is as follows:
\begin{equation*}
    \left\{\begin{aligned}
        &-\Delta U(X)=f(X)\cdot U(X)^{-\gamma}&\mbox{ in }&B_{2}\cap Cone_{\Sigma}\\
        &U(Y)=0&\mbox{ on }&\partial(B_{2}\cap Cone_{\Sigma})
    \end{aligned}\right..
\end{equation*}

By the maximal principle, we can replace $U(X)$ satisfying $-\Delta U=f(X)\cdot U^{-\gamma}$ with $U_{\lambda}$ and $U_{\Lambda}$, such that 
\begin{equation*}
    U_{\lambda}(Y)=U_{\Lambda}(Y)=0\mbox{ on }\partial(B_{2}\cap Cone_{\Sigma})
\end{equation*}
and in $B_{2}\cap Cone_{\Sigma}$ they satisfy
\begin{equation*}
    -\Delta U_{\lambda}=\lambda\cdot U_{\lambda}^{-\gamma},\quad-\Delta U_{\Lambda}=\Lambda\cdot U_{\Lambda}^{-\gamma}.
\end{equation*}
Then $U(X)$ is bounded between $U_{\lambda}$ and $U_{\Lambda}$. Therefore, it suffices to prove Theorem~\ref{thm. growth rate} when $f(X)$ is a constant (for example $f(X)\equiv1$).

To wrap up, we focus on studying the boundary growth rate for $U(X)$ satisfying
\begin{equation}\label{eq. U(X) special}
    \left\{\begin{aligned}
        &-\Delta U(X)=U(X)^{-\gamma}&\mbox{ in }&B_{1}\cap Cone_{\Sigma}\\
        &U(Y)=0&\mbox{ on }&\partial(B_{1}\cap Cone_{\Sigma})
    \end{aligned}\right..
\end{equation}
As was previously explained, Theorem~\ref{thm. growth rate} and Theorem~\ref{thm. growth rate, (b) improved} will be true for general $u(X)$'s, provided that they are true for the specific $U(X)$ chosen in \eqref{eq. U(X) special}.

We let $V: Cone_{\Sigma}\cap B_{1}\to\mathbb{R}$ be defined as
\begin{equation}\label{eq. V(X) moving rescale}
    V(X)=2^{\frac{2}{1+\gamma}}U(X/2)-U(X).
\end{equation}
We claim that $V(X)\geq0$ vanishes at $B_{1}\cap\partial Cone_{\Sigma}$ and $\Delta V(X)\geq0$.

In fact, it is obvious to see $V(X)=0$ at $B_{1}\cap\partial Cone_{\Sigma}$. Besides, as $2^{\frac{2}{1+\gamma}}U(X/2)$ is also a solution to the "SLEF" $-\Delta u=u^{-\gamma}$ with non-negative boundary value on $\partial(B_{1}\cap Cone_{\Sigma})$, we infer that
\begin{equation}\label{eq. rescale is larger}
    2^{\frac{2}{1+\gamma}}U(X/2)\geq U(X).
\end{equation}
Thus, the monotonicity of the right-hand side plus \eqref{eq. rescale is larger} will imply $\Delta V\geq0$.

Now let us prove Theorem~\ref{thm. growth rate}.
\begin{proof}[Proof of Theorem~\ref{thm. growth rate} (a)-(c)]
    Using Lemma~\ref{lem. star shape cone}, we may assume without loss of generality that $\Sigma$ is star-shaped and Lipschitz, and that $Cone_{\Sigma}$ is a Lipschitz graph. We mainly study the special example $U(X)$ defined in \eqref{eq. U(X) special}. For the solution $v$ satisfying \eqref{eq. v general solution in a cone}, its boundary estimates following from comparing $v(X)$ with $U(X)$ with the help of Theorem~\ref{thm. Harnack Comparable, general Lipschitz domain}.
    
    From the discussion above, we already have that the auxiliary function $V(x)$ constructed in \eqref{eq. V(X) moving rescale} is subharmonic, positive and vanishes at $\partial Cone_{\Sigma}$. As $\partial Cone_{\Sigma}$ is a Lipschitz graph with some uniform Lipschitz norm $L$, we apply the boundary Harnack principle for the Laplace equation (see Lemma~\ref{lem. classical Kemper} or \cite{K72}), and obtain
    \begin{equation*}
        V(X)\leq C\cdot H_{\Sigma}(X)\mbox{ in }B_{1/2}\cap Cone_{\Sigma},
    \end{equation*}
    for $H_{\Sigma}(X)$ defined in \eqref{eq. H Sigma, homogeneous harmonic in a cone}. In other words
    \begin{equation}\label{eq. w growth rate is phi}
        V(X)\leq C|X|^{\phi_{\Sigma}}\quad\mbox{in }B_{1/2}\cap Cone_{\Sigma}.
    \end{equation}
    
    We set $r_{k}=2^{-k}$ for $k\geq1$ and
    \begin{equation*}
        A_{k}=r_{k}^{-\frac{2}{1+\gamma}}\max_{Cone_{\Sigma}\cap B_{r_{k}}}U(X),\quad B_{k}=r_{k}^{-\frac{2}{1+\gamma}}\max_{Cone_{\Sigma}\cap B_{r_{k}}}V(X).
    \end{equation*}
    For each $k\geq1$ and any $X\in Cone_{\Sigma}\cap B_{r_{k}}$, we have $\frac{X}{2}\in Cone_{\Sigma}\cap B_{r_{k+1}}$. By dividing both sides of \eqref{eq. V(X) moving rescale} by $r_{k}^{\frac{2}{1+\gamma}}$, we obtain that
    \begin{equation*}
        \frac{U(X/2)}{r_{k+1}^{\frac{2}{1+\gamma}}}=\frac{2^{\frac{2}{1+\gamma}}U(X/2)}{r_{k}^{\frac{2}{1+\gamma}}}=\frac{U(X)}{r_{k}^{\frac{2}{1+\gamma}}}+\frac{V(X)}{r_{k}^{\frac{2}{1+\gamma}}}\leq A_{k}+B_{k}.
    \end{equation*}
    By the arbitrariness of $X$ and by \eqref{eq. w growth rate is phi}, we see
    \begin{equation*}
        A_{k+1}-A_{k}\leq B_{k}\leq C\cdot r_{k}^{\phi_{\Sigma}-\frac{2}{1+\gamma}}=C\cdot2^{(\frac{2}{1+\gamma}-\phi_{\Sigma})k}.
    \end{equation*}
    We then sum up the inequality above for $1\leq i\leq k-1$, and obtain
    \begin{equation*}
        A_{k}\leq A_{1}+C\sum_{i=1}^{k-1}2^{(\frac{2}{1+\gamma}-\phi_{\Sigma})i}.
    \end{equation*}

    In the sub-critical case, the common ratio
    \begin{equation*}
        q=2^{(\frac{2}{1+\gamma}-\phi_{\Sigma})}<1,
    \end{equation*}
    so $A_{k}\leq A_{1}+C(q^{1}+q^{2}+q^{3}+\cdots)\leq C$, which depends on the difference between $\phi_\Sigma$ and $\frac{2}{1+\gamma}$. Then,
    \begin{equation*}
        U(X)\leq C r_{k}^{\frac{2}{1+\gamma}}\quad\mbox{in }Cone_{\Sigma}\cap B_{r_{k}}.
    \end{equation*}

    For the solution $v$ to \eqref{eq. v general solution in a cone}, it follows from Theorem~\ref{thm. Harnack Comparable, general Lipschitz domain} that
    \begin{equation*}
        v(X)\leq C\|v\|_{L^{\infty}(Cone_{\Sigma}\cap B_{2})}\cdot U(X)\leq C\|v\|_{L^{\infty}(Cone_{\Sigma}\cap B_{2})}\cdot|X|^{\frac{2}{1+\gamma}}\mbox{ in }Cone_{\Sigma}\cap B_{1}.
    \end{equation*}
    By \eqref{eq. v pointwise vs u pointwise} and \eqref{eq. v L infty vs u L infty}, we arrive at the following boundary estimate for $u(X)$:
    \begin{equation*}
        u(X)\leq C\|u\|_{L^{\infty}(\mathcal{GC}_{3\sqrt{L^{2}+1}})}\cdot|X|^{\frac{2}{1+\gamma}}\mbox{ in }\mathcal{GC}_{\sqrt{L^{2}+1}},
    \end{equation*}
    which is the desired upper bound estimate in Theorem~\ref{thm. growth rate} (a).
    
    Similar computations can be conducted for the critical and super-critical cases, using different versions of the summation formula for geometric sequences. We omit the details. Then, we have completed the proof of Theorem~\ref{thm. growth rate} (a)-(c).
\end{proof}

Similar to the proof of Theorem~\ref{thm. growth rate} (a)-(c), when proving parts (d) and (e), the domain inclusion $Cone_{\Sigma}\subseteq\mathcal{GC}_{3R}$ allows us to assume without loss of generality that $Cone_{\Sigma}$ and $\mathcal{GC}_{3R}$ are identical in $B_{2}$.
\begin{proof}[Proof of Theorem~\ref{thm. growth rate} (d)(e)] We first prove the optimality part (d).
\begin{itemize}
    \item In the sub-critical case, the optimality simply follows from Lemma~\ref{lem. non degenerate}.
    \item In the super-critical case, let $v$ be the harmonic replacement of $u$ in $B_{1}\cap Cone_{\Sigma}$, then
    \begin{equation*}
        \mbox{as }-\Delta v=0\leq v^{-\gamma},\quad\mbox{we have }u\geq v.
    \end{equation*}
    Again we apply the boundary Harnack principle (Lemma~\ref{lem. classical Kemper} or \cite{K72}) for the Laplace equation and obtain that
    \begin{equation*}
        v(X)\geq C^{-1}\cdot H_{\Sigma}(X)\mbox{ in }B_{1/2}\cap Cone_{\Sigma}.
    \end{equation*}
    In particular on the ray $t\vec{e_{n}}$, we have
    \begin{equation*}
        u(t\vec{e_{n}})\geq v(t\vec{e_{n}})\geq C^{-1}\cdot H_{\Sigma}(t\vec{e_{n}})\geq C^{-1}|t|^{\phi_{\Sigma}}.
    \end{equation*}
\end{itemize}
Hence, we have completed the proof of part (d).

Now, we give the proof of part (e). We see that when $K>0$ is sufficiently small,
\begin{equation*}
    \underline{U}(X)=K(\ln{\frac{1}{|X|}})^{\phi_{\Sigma}/2}H_{\Sigma}(X)
\end{equation*}
is a lower solution of the "SLEF" supported in $Cone_{\Sigma}\cap B_{1/e}$. In fact, we need to recall the fact that $H_{\Sigma}$ is harmonic, homogeneous with degree $\phi_{\Sigma}$, and that
\begin{equation*}
    H_{\Sigma}(X)\leq|X|^{\phi_{\Sigma}}\quad(\mbox{since we assume}\max_{B_{1}\cap Cone_{\Sigma}}H_{\Sigma}(X)=1).
\end{equation*}
Besides, as $\phi_{\Sigma}=\frac{2}{1+\gamma}$ with $\gamma>0$, we infer that $\phi_{\Sigma}<2$. Then for any $X\in Cone_{\Sigma}\cap B_{1/e}$,
\begin{align*}
    -\Delta\underline{U}(X)=&-K\cdot H_{\Sigma}(X)\Delta(\ln{\frac{1}{|X|}})^{\phi_{\Sigma}/2}-2K\frac{\partial(\ln{\frac{1}{|X|}})^{\phi_{\Sigma}/2}}{\partial r}\cdot\frac{\partial H_{\Sigma}(X)}{\partial r}\\
    =&-K\cdot H_{\Sigma}(X)\Big\{\frac{\phi_{\Sigma}}{2|X|^{2}}(\ln{\frac{1}{|X|}})^{\frac{\phi_{\Sigma}}{2}-1}+\frac{\phi_{\Sigma}}{2|X|^{2}}(\frac{\phi_{\Sigma}}{2}-1)(\ln{\frac{1}{|X|}})^{\frac{\phi_{\Sigma}}{2}-2}\\
    &-\frac{(n-1)\phi_{\Sigma}}{2|X|^{2}}(\ln{\frac{1}{|X|}})^{\frac{\phi_{\Sigma}}{2}-1}\Big\}+2K\cdot\frac{\phi_{\Sigma}}{2|X|}(\ln{\frac{1}{|X|}})^{\frac{\phi_{\Sigma}}{2}-1}\cdot\frac{\phi_{\Sigma}}{|X|}H_{\Sigma}(X)\\
    =&\frac{K\phi_{\Sigma}\cdot H_{\Sigma}(X)}{2|X|^{2}}(\ln{\frac{1}{|X|}})^{\frac{\phi_{\Sigma}}{2}-1}\Big\{(n-2+2\phi_{\Sigma})+(1-\frac{\phi_{\Sigma}}{2})(\ln{\frac{1}{|X|}})^{-1}\Big\}\\
    \leq&\frac{1}{2}K\phi_{\Sigma}(n-1+\frac{3}{2}\phi_{\Sigma})|X|^{\phi_{\Sigma}-2}(\ln{\frac{1}{|X|}})^{\frac{\phi_{\Sigma}}{2}-1}.
\end{align*}
On the other hand, as $\gamma=\frac{2}{\phi_{\Sigma}}-1>0$ in the critical case, we have
\begin{equation*}
    \underline{U}(X)^{-\gamma}=K^{-\gamma}H_{\Sigma}(X)^{1-\frac{2}{\phi_{\Sigma}}}(\ln{\frac{1}{|X|}})^{\frac{\phi_{\Sigma}}{2}-1}\geq K^{-\gamma}|X|^{\phi_{\Sigma}-2}(\ln{\frac{1}{|X|}})^{\frac{\phi_{\Sigma}}{2}-1}.
\end{equation*}
Therefore, as long as we choose a small $K\ll1$, we have
\begin{equation*}
    -\Delta\underline{U}(X)\ll\underline{U}(X)^{-\gamma}\quad\mbox{in }Cone_{\Sigma}\cap B_{1/e}.
\end{equation*}
It then provides a lower bound of a solution $U(X)$ defined in $Cone_{\Sigma}\cap B_{1/e}$ having the same boundary data. Then by Theorem~\ref{thm. Harnack Comparable, general Lipschitz domain}, the lower bound in part (e) holds for all other solutions as well.\end{proof}

We end this subsection by giving a proof of Corollary~\ref{cor. of thm 1.2}.
\begin{proof}[Proof of Corollary~\ref{cor. of thm 1.2}]
    From Theorem~\ref{thm. growth rate} (a)-(c), we obtain a pointwise $C^{\mu}$ continuity of the solution $u$ near the boundary for some sufficiently small $\mu$ (depending on $[\partial\Omega]_{C^{0,1}}$ and $\gamma$). Besides, from Theorem~\ref{thm. growth rate} (d)(e), we can control the growth rate of $f(X)u^{-\gamma}$ from above near the boundary. Then, for each point $X\in\Omega$ whose distance to $\partial\Omega$ equals $r$, we apply the interior Schauder estimate in $B_{r/2}(X)$, and obtain that
    \begin{equation}\label{eq. interior C mu estimate}
        \|u\|_{C^{\mu}(B_{r/4}(X))}\leq C,
    \end{equation}
    where $\mu$ and $C$ depend only on $\Omega$ and $\gamma$. Then, a covering argument yields that $u\in C^{\mu}(\overline{\Omega})$.
    
    We now prove part (a), where $\gamma>1$. Observe that as $\frac{2}{1+\gamma}<1$, there exists $\delta=\delta(\gamma)>0$ such that the cone
    \begin{equation*}
        Cone_{\Sigma}:=\{x_{n}\geq-\delta|x'|\}
    \end{equation*}
    satisfies $\phi_{\Sigma}>\frac{2}{1+\gamma}$. Since $\Gamma$ is a $C^{1}$ or a convex boundary, after a suitable rotation, we can find some $r>0$ such that
    \begin{equation*}
        (\Omega-Y)\cap B_{r}\subseteq Cone_{\Sigma},\quad\mbox{for all }Y\in B_{r}\cap\partial\Omega.
    \end{equation*}
    Here, $(\Omega-Y)$ refers to the translation of $\Omega$ by the vector $-Y$. Applying Theorem~\ref{thm. growth rate} (a) gives that
    \begin{equation*}
        u(Y+t\vec{e_{n}})\leq C t^{\frac{2}{1+\gamma}}
    \end{equation*}
    for sufficiently small $t$, while Theorem~\ref{thm. growth rate} (d) yields that
    \begin{equation*}
        u(Y+t\vec{e_{n}})\geq c t^{\frac{2}{1+\gamma}}.
    \end{equation*}
    Consequently, \eqref{eq. interior C mu estimate} holds with $\mu=\frac{2}{1+\gamma}$, which completes the proof of Corollary~\ref{cor. of thm 1.2} (a).
    
    Part (b1) and (b2) follow analogously. For (b1), we choose a suitable inner cone
    \begin{equation*}
        Cone_{\Sigma}:=\{x_{n}\geq\delta|x'|\}\mbox{ for }\delta>0.
    \end{equation*}
    For (b2), we let the inner cone be $Cone_{\Sigma}=\mathbb{R}^{n}_{+}$. In both cases, we have $\phi_{\Sigma}<\frac{2}{1+\gamma}$ and
    \begin{equation*}
        Cone_{\Sigma}\cap B_{r}\subseteq(\Omega-Y),\quad\mbox{for all }Y\in B_{r}\cap\partial\Omega.
    \end{equation*}
    Using Theorem~\ref{thm. growth rate} (d), we obtain the estimates \eqref{eq. estimate b1 in cor 1.1} and \eqref{eq. estimate b2 in cor 1.1}. Here, the exponent $\alpha$ in \eqref{eq. estimate b1 in cor 1.1} can be made arbitrarily close to $1$ as $\delta\to0$.
\end{proof}
\subsection{Proof of Theorem~\ref{thm. growth rate, (b) improved}}
Now we improve the estimate in Theorem~\ref{thm. growth rate} part (b), given the existence of a continuous solution to \eqref{eq. growth rate special condition in critical case}. One more time as a reminder, when proving Theorem~\ref{thm. growth rate, (b) improved}, we assume that $\mathcal{GC}_{3R}$ (with $R=\sqrt{L^{2}+1}$) and $Cone_{\Sigma}$ are identical in $B_{2}$ and $\Sigma$ is star-shaped.

For the sake of convenience, we assume $f(X)\equiv1$ just like in the previous subsection. Besides, we let $\Sigma$ be fixed, and without causing much confusion, we simply write
\begin{equation*}
    \phi=\phi_{\Sigma}=\frac{2}{1+\gamma}.
\end{equation*}

We denote $r_{k}=2^{-k}$ for $k\geq1$. Let $U(X)$ be a solution to $-\Delta u=u^{-\gamma}$ defined in $Cone_{\Sigma}\cap B_{1}$, such that for a sufficiently large $T$ (to be specified in Lemma~\ref{lem. Claim of Ak iteration} below),
\begin{equation*}
    U(\frac{1}{2}\vec{e_{n}})\geq2^{-\phi}+T\cdot H_{\Sigma}(\frac{1}{2}\vec{e_{n}}).
\end{equation*}
In fact, for every $T$, such a special solution $U(X)$ exists. To see this, we just require $U(X)$ to solve the problem:
\begin{equation}\label{eq. special U(X) in Thm 1.3}
    \left\{\begin{aligned}
        &-\Delta U(X)=U(X)^{-\gamma}&\mbox{ in }&Cone_{\Sigma}\cap B_{1}\\
        &U(Y)=(T+\frac{2^{-\phi}}{H_{\Sigma}(\frac{1}{2}\vec{e_{n}})})\cdot H_{\Sigma}(Y)&\mbox{ on }&\partial(Cone_{\Sigma}\cap B_{1})
    \end{aligned}\right..
\end{equation}
Since $U(X)$ is super-harmonic, we then have
\begin{equation*}
    U(\frac{1}{2}\vec{e_{n}})\geq(T+\frac{2^{-\phi}}{H_{\Sigma}(\frac{1}{2}\vec{e_{n}})})H_{\Sigma}(\frac{1}{2}\vec{e_{n}})=2^{-\phi}+T\cdot H_{\Sigma}(\frac{1}{2}\vec{e_{n}}).
\end{equation*}
We then define a sequence with $A_{1}\geq T$ as follows:
\begin{equation}\label{eq. Ak in Thm 1.3}
    A_{k}=\max_{2^{-k-1}\leq|X|\leq2^{-k}}\frac{U(X)-2^{-k\phi}k^{-\gamma\phi/2}}{T\cdot H_{\Sigma}(X)}.
\end{equation}
The following claim is the key to prove Theorem~\ref{thm. growth rate, (b) improved}.
\begin{lemma}\label{lem. Claim of Ak iteration}
    Let $T$ be sufficiently large. Construct the special solution $U(X)$ as in \eqref{eq. special U(X) in Thm 1.3}, and define quantities $A_{k}$ as in \eqref{eq. Ak in Thm 1.3}. For every $k\geq1$, if $A_{k}\geq(k+1)^{\phi/2}$, then
    \begin{equation*}
        A_{k+1}\leq A_{k}+A_{k}^{-\gamma}+k^{\frac{-\gamma}{1+\gamma}}.
    \end{equation*}
    Otherwise if $A_{k}\leq(k+1)^{\phi/2}$, then
    \begin{equation*}
        A_{k+1}\leq(k+1)^{\phi/2}+2.
    \end{equation*}
\end{lemma}
\begin{proof}
    For every $k\geq1$, we let
    \begin{equation*}
        U_{k}(X)=2^{k\phi}U(2^{-k}X)\quad\mbox{defined in }Cone_{\Sigma}\cap B_{1}
    \end{equation*}
    to be an invariant scaling of $U(X)$. Assume that $A_{k}\geq(k+1)^{\phi/2}$, and let $V_{k}(X)$ be the solution to
    \begin{equation*}
        \left\{\begin{aligned}
            &-\Delta V_{k}=V_{k}^{-\gamma}&\mbox{ in }&Cone_{\Sigma}\cap B_{1}\\
            &V_{k}(Y)=k^{-\gamma\phi/2}+A_{k}T\cdot H_{\Sigma}(Y)&\mbox{ on }&Cone_{\Sigma}\cap\partial B_{1}\\
            &V_{k}(Y)=k^{-\gamma\phi/2}\cdot\max\{10|Y|-9,0\}&\mbox{ on }&B_{1}\cap\partial Cone_{\Sigma}
        \end{aligned}\right..
    \end{equation*}
    We then have
    \begin{equation*}
        V_{k}(X)\geq U_{k}(X)\mbox{ in }Cone_{\Sigma}\cap B_{1}.
    \end{equation*}
    Notice that $A_{k}T\cdot H_{\Sigma}(X)$ is harmonic and less than $V_{k}(X)$ at $\partial(Cone_{\Sigma}\cap B_{1})$, we have
    \begin{equation*}
        A_{k}T\cdot H_{\Sigma}(X)\leq V_{k}(X)
    \end{equation*}
    as well in the interior. Therefore, we can bound $V_{k}$ from above by $\overline{V_{k}}(X)$ satisfying
    \begin{equation*}
        \left\{\begin{aligned}
            &-\Delta\overline{V_{k}}=(A_{k}T\cdot H_{\Sigma}(X))^{-\gamma}&\mbox{ in }&Cone_{\Sigma}\cap B_{1}\\
            &\overline{V_{k}}(Y)=k^{-\gamma\phi/2}+A_{k}T\cdot H_{\Sigma}(Y)&\mbox{ on }&Cone_{\Sigma}\cap\partial B_{1}\\
            &\overline{V_{k}}(Y)=k^{-\gamma\phi/2}\cdot\max\{10|Y|-9,0\}&\mbox{ on }&B_{1}\cap\partial Cone_{\Sigma}
        \end{aligned}\right..
    \end{equation*}
    Besides, let $\underline{V_{k}}$ be the harmonic replacement of $V_{k}$, then
    \begin{equation*}
        \underline{V_{k}}(X)=A_{k}T\cdot H_{\Sigma}(X)+k^{-\gamma\phi/2}W_{1}(X),
    \end{equation*}
    for $W_{1}(X)$ being a fixed harmonic function defined in $Cone_{\Sigma}\cap B_{1}$ satisfying
    \begin{equation*}
        \left\{\begin{aligned}
            &-\Delta W_{1}(X)=0&\mbox{ in }&Cone_{\Sigma}\cap B_{1}\\
            &W_{1}(Y)=\max\{10|Y|-9,0\}&\mbox{ on }&B_{1}\cap\partial Cone_{\Sigma}
        \end{aligned}\right..
    \end{equation*}It then follows that
    \begin{equation*}
        -\Delta(\overline{V_{k}}-\underline{V_{k}})=A_{k}^{-\gamma}T^{-\gamma}H_{\Sigma}(X)^{-\gamma},\quad\mbox{and }\overline{V_{k}}-\underline{V_{k}}=0\mbox{ at }\partial(Cone_{\Sigma}\cap B_{1}).
    \end{equation*}
    
    As we assume the existence of $-\Delta^{-1}(H_{\Sigma}^{-\gamma})$, we let $W_{2}(X)$ be the solution to
    \begin{equation*}
        \left\{\begin{aligned}
            &-\Delta W_{2}(X)=H_{\Sigma}(X)^{-\gamma}&\mbox{ in }&Cone_{\Sigma}\cap B_{1}\\
            &W_{2}=0&\mbox{ on }&\partial(Cone_{\Sigma}\cap B_{1})
        \end{aligned}\right..
    \end{equation*}
    We then use the upper bound $\overline{V_{k}}(X)$ and have
    \begin{equation*}
        U_{k}(X)\leq A_{k}T\cdot H_{\Sigma}(X)+k^{-\gamma\phi/2}W_{1}(X)+A_{k}^{-\gamma}T^{-\gamma}W_{2}(X).
    \end{equation*}
    If $T$ is sufficiently large, we can guarantee that for $X\in Cone_{\Sigma}\cap(B_{1/2}\setminus B_{1/4})$,
    \begin{equation*}
        W_{1}(X)\leq T\cdot H_{\Sigma}(X),\quad T^{-\gamma}W_{2}(X)\leq T\cdot H_{\Sigma}(X)+2^{-\phi}.
    \end{equation*}
    It then follows that on $Cone_{\Sigma}\cap(B_{1/2}\setminus B_{1/4})$,
    \begin{equation*}
        U_{k}(X)\leq(A_{k}+A_{k}^{-\gamma}+k^{-\gamma\phi/2})T\cdot H_{\Sigma}(X)+2^{-\phi}A_{k}^{-\gamma}.
    \end{equation*}
    Notice that
    \begin{equation*}
        2^{-\phi}A_{k}^{-\gamma}\leq2^{-\phi}(k+1)^{-\gamma\phi/2},
    \end{equation*}
    so we rescale back to $B_{2^{-k}}$ and conclude that
    \begin{equation*}
        U(X)\leq(A_{k}+A_{k}^{-\gamma}+k^{-\gamma\phi/2})T\cdot H_{\Sigma}(X)+2^{-(k+1)\phi}(k+1)^{-\gamma\phi/2},
    \end{equation*}
    which means $A_{k+1}\leq A_{k}+A_{k}^{-\gamma}+k^{-\gamma\phi/2}$.

    The second case where $A_{k}\leq(k+1)^{\phi/2}$ can be proven similarly by always replacing $A_{k}$ with $(k+1)^{\phi/2}$.
\end{proof}

We are now in a position to prove Theorem~\ref{thm. growth rate, (b) improved}.
\begin{proof}[Proof of Theorem~\ref{thm. growth rate, (b) improved}]
Let $C$ be sufficiently large, and suppose that $A_{m}\leq C m^{\phi/2}=C m^{\frac{1}{1+\gamma}}$ for some $m\geq1$, then by Lemma~\ref{lem. Claim of Ak iteration} we have (also recall that $\frac{\phi}{2}-1=\frac{-\gamma}{1+\gamma}$):
\begin{equation*}
    A_{m+1}\leq\left\{\begin{aligned}
        &C m^{\phi/2}+(m+1)^{\frac{-\gamma}{1+\gamma}}+m^{\frac{-\gamma}{1+\gamma}}\leq C(m+1)^{\phi/2},&\mbox{if }&A_{m}\geq(m+1)^{\phi/2}\\
        &(m+1)^{\phi/2}+2\leq C(m+1)^{\phi/2},&\mbox{if }&A_{m}\leq(m+1)^{\phi/2}
    \end{aligned}\right..
\end{equation*}
The inductive argument shows that $A_{k}\leq C k^{\phi/2}$ for all $k\geq1$. Therefore,
\begin{equation*}
    U(X)=O\Big(|X|^{\phi}(\ln{\frac{1}{|X|}})^{\phi/2}\Big).
\end{equation*}
We then use Theorem~\ref{thm. Harnack Comparable, general Lipschitz domain} to obtain the same estimate (with a different factor in front) for general $u(X)$'s.
\end{proof}

\section{Continuity of the ratio \texorpdfstring{$\ds\frac{u}{v}$}{Lg}}\label{sec. continuity of the ratio u/v}
In this section, our goal is to show Theorem~\ref{thm. limit is 1} and Theorem~\ref{thm. limit is continuous when obtuse}.

\subsection{Proof of Theorem~\ref{thm. limit is 1}}
The key to proof Theorem~\ref{thm. limit is 1} is to somehow iterate the proof of Theorem~\ref{thm. Harnack Comparable, general Lipschitz domain}. We denote $r_{k}=3^{1-k}R$, and define a sequence $w_{k}$ with
\begin{equation*}
    w_{k}=u-(1-\sigma_{k})v,
\end{equation*}
where $\sigma_{k}\to0$ is to be given later. If we can show that $w_{k}\geq0$ in $\mathcal{GC}_{r_{k}}$, then $\frac{u}{v}-1\geq-\sigma_{k}$. We will always set $\sigma_{0}=Q$, then by the assumptions in Theorem~\ref{thm. limit is 1}, $w_{0}\geq0$ in $\mathcal{GC}_{r_{0}}=\mathcal{GC}_{3R}$.

In fact, we will show that $\sigma_{k}$ decays like a geometric sequence in the sub-critical case, while like a negative power sequence in the critical case.

On the other hand, one can not guarantee that $\sigma\to0$ in the super-critical case. This serves as a partial explanation of Theorem~\ref{thm. ratio not continuous}.

The following lemma is essential in the proof. It describes how a solution to the "negative-eigenvalue equation" bends down in the interior.
\begin{lemma}[an ABP-type estimate]\label{lem. negative eigenvalue equation}
    Assume that $0\leq u\leq v$ in $B_{1}$ and
    \begin{equation*}
        \Delta v\leq0,\quad\Delta u\geq t u
    \end{equation*}
    for some $t\in(0,1/2)$. Then there is $c=c(n)$ such that in $B_{1/2}$ we have
    \begin{equation*}
        u\leq(1-ct)v.
    \end{equation*}
\end{lemma}
\begin{proof}
    Without loss of generality, we can assume $-\Delta v=0$ and hence
    \begin{equation*}
        v(x)\geq c_{1}v(0)\mbox{ in }B_{7/8}.
    \end{equation*}
    Since $(v-u)\geq0$ is super-harmonic, by applying the weak Harnack principle, it suffices to show that there exist small constants $c_{2}$, $c_{3}$ such that
    \begin{equation}\label{eq. density not small}
        \Big|\{v-u>c_{2}t v(0)\}\cap B_{3/4}\Big|\geq c_{3}.
    \end{equation}
    In fact, we can argue by contradiction. Suppose otherwise, then in particular,
    \begin{equation*}
        \Big|\{u>c_{1}v(0)\}\cap B_{3/4}\Big|\geq\Big|\{v-u<2c_{1}t v(0)\}\cap B_{3/4}\Big|\geq(1-c_{3})\big|B_{3/4}\big|.
    \end{equation*}
    We denote $E=\{u>c_{1}v(0)\}\cap B_{3/4}$, then in $B_{3/4}$,
    \begin{equation*}
        \Delta(v-u)\leq-c_{1}t v(0)\chi_{E}.
    \end{equation*}
    When $c_{3}$ is sufficiently small, then by the ABP estimate, we conclude that
    \begin{equation*}
        \min_{B_{1/2}}(v-u)\geq c_{1}c_{5}t v(0).
    \end{equation*}
    By reassigning $c_{2}=c_{1}c_{5}$ and $c_{3}=\big|B_{1/2}\big|$, we still have \eqref{eq. density not small}, and thus have reached a contradiction.
\end{proof}

We first consider the sub-critical case. We assume that $\mathcal{GC}_{3R}$ is included in $Cone_{\Sigma}$ such that $\frac{2}{1+\gamma}<\phi_{\Sigma}$. Notice that by Theorem~\ref{thm. growth rate} (a),
\begin{equation*}
    \|v\|_{L^{\infty}(\mathcal{GC}_{3^{1-k}R})}\leq C(n,L,\gamma,\lambda,\Lambda,\phi_\Sigma)\|v\|_{L^{\infty}(\mathcal{GC}_{r_{k}})}\cdot3^{-\frac{2k}{1+\gamma}}.
\end{equation*}

\begin{proof}[Proof of Theorem~\ref{thm. limit is 1} (a)]
    Suppose that $\sigma_{k}$ is chosen such that $w_{k}\geq0$, namely
    \begin{equation*}
        (v-u)\leq\sigma_{k} v\quad\mbox{in }\mathcal{GC}_{r_{k}}=\mathcal{GC}_{3^{1-k}R}.
    \end{equation*}
    It then implies
    \begin{equation*}
        0\leq\max\{v-u,0\}\leq\sigma_{k} v\quad\mbox{in }\mathcal{GC}_{r_{k}}=\mathcal{GC}_{3^{1-k}R}.
    \end{equation*}
    Using the convexity of $x^{-\gamma}$ we have
    \begin{equation*}
        \Delta\max\{v-u,0\}\geq-f(X)\cdot(v^{-\gamma}-u^{-\gamma})\cdot\chi_{\{u<v\}}\geq\frac{\lambda\gamma}{\|v\|_{L^{\infty}(\mathcal{GC}_{r_{k}})}^{\gamma+1}}\max\{v-u,0\}\mbox{ in }\mathcal{GC}_{r_{k}}.
    \end{equation*}
    By Theorem~\ref{thm. growth rate} (a), we have
    \begin{equation*}
        \|v\|_{L^{\infty}(\mathcal{GC}_{r_{k}})}^{1+\gamma}\leq\frac{C}{9^{k}}\|v\|_{L^{\infty}(\mathcal{GC}_{3R})}^{1+\gamma}.
    \end{equation*}
    By Lemma~\ref{lem. negative eigenvalue equation} we conclude that in $\mathcal{SC}_{\frac{2}{3}r_{k}}=\mathcal{SC}_{2r_{k+1}}$,
    \begin{equation}\label{eq. ABP outcome a}
        v-u\leq(1-c\frac{r_{k}^{2}\gamma\lambda}{\|v\|_{L^{\infty}(\mathcal{GC}_{r_{k}})}^{\gamma+1}})\sigma_{k}v\leq(1-c_{1})\sigma_{k}v,\quad\mbox{for }c_{1}=\frac{c\gamma\lambda R^{2}}{C\|v\|_{L^{\infty}(\mathcal{GC}_{3R})}^{\gamma+1}}.
    \end{equation}

    Now we consider
    \begin{equation*}
        w_{k+1}=u-(1-\sigma_{k+1})v
    \end{equation*}
    with $\sigma_{k+1}$ yet to be chosen. It suffices to show $w_{k+1}\geq0$ in $\mathcal{GC}_{r_{k+1}}$. First we notice that in $\mathcal{GC}_{2r_{k+1}}\subseteq\mathcal{GC}_{r_{k}}$,
    \begin{equation*}
        w_{k+1}\geq (1-\sigma_{k})v-(1-\sigma_{k+1})v\geq-(\sigma_{k}-\sigma_{k+1})v,
    \end{equation*}
    while in the suspended cylinder $\mathcal{SC}_{2r_{k+1}}$, we use the estimate \eqref{eq. ABP outcome a} and obtain
    \begin{equation*}
        w_{k+1}\geq\Big(\sigma_{k+1}-(1-c_{1})\sigma_{k}\Big)v.
    \end{equation*}
    We can then choose $\sigma_{k+1}=(1-c_{1}c_{2})\sigma_{k}$ for a small $c_{2}$, meaning
    \begin{align*}
        w_{k+1}\geq-c_{1}c_{2}\sigma_{k}v\quad&\mbox{in }\mathcal{GC}_{2r_{k+1}},\\
        w_{k+1}\geq c_{1}(1-c_{2})\sigma_{k}v\quad&\mbox{in }\mathcal{SC}_{2r_{k+1}}.
    \end{align*}
    By choosing $c_{2}$ small such that
    \begin{equation*}
        \frac{c_{2}}{1-c_{2}}\leq\frac{1}{C_{3}M}\leq\frac{\min_{\mathcal{SC}_{2r_{k+1}}}v}{M\max_{\mathcal{GC}_{2r_{k+1}}}v},
    \end{equation*}
    where $C_{3}$ is the same as that in Lemma~\ref{lem. u is small near the boundary} so that
    \begin{equation*}
        \max_{\mathcal{GC}_{2r_{k+1}}}v\leq C_{3}\min_{\mathcal{SC}_{2r_{k+1}}}v.
    \end{equation*}
    It then follows that
    \begin{equation*}
        \min_{\mathcal{SC}_{2r_{k+1}}}w_{k+1}\geq-M\min_{\mathcal{GC}_{2r_{k+1}}}w_{k+1}.
    \end{equation*}
    Moreover, we can argue similarly as the proof of Theorem~\ref{thm. Harnack Comparable, general Lipschitz domain} that
    \begin{equation*}
        \Delta w_{k+1}\leq\frac{C}{dist(X,\Gamma)^{2}}\max\{w_{k+1},0\},
    \end{equation*}
    then we apply Lemma~\ref{lem. ensure positive} and conclude that $w_{k+1}\geq0$ in $\mathcal{GC}_{r_{k+1}}$.

    Recall that we always choose $\sigma_{0}=Q$, it the follows that $\sigma_{k}\leq(1-c_{1}c_{2})^{k}Q$. In other words,
    \begin{equation*}
        \ds\sigma_{k}\leq Q\cdot(\frac{r_{k}}{R})^{\epsilon},\quad\mbox{for all }k\geq0.
    \end{equation*}
    Here,
    \begin{equation*}
        \epsilon\sim c_{1}c_{2}\sim\frac{R^{2}}{\|v\|_{L^{\infty}(\mathcal{GC}_{3R})}^{1+\gamma}}.
    \end{equation*}
\end{proof}

Next, we similarly consider the critical case. We assume that $\mathcal{GC}_{3R}$ is included in $Cone_{\Sigma}$ such that $\frac{2}{1+\gamma}=\phi_{\Sigma}$ and that \eqref{eq. growth rate special condition in critical case} is solvable. Notice that by Theorem~\ref{thm. growth rate, (b) improved},
\begin{equation*}
    \|v\|_{L^{\infty}(\mathcal{GC}_{3^{1-k}R})}\leq C(n,L,\gamma,\lambda,\Lambda,\phi_\Sigma)\|v\|_{L^{\infty}(\mathcal{GC}_{r_{k}})}\cdot3^{-\frac{2k}{1+\gamma}}k^{\frac{1}{1+\gamma}}.
\end{equation*}

\begin{proof}[Proof of Theorem~\ref{thm. limit is 1} (b)]
    We similarly assume $\sigma_{k}$ is chosen such that $w_{k}\geq0$ in $\mathcal{GC}_{r_{k}}$. This time, as
    \begin{equation*}
        \|v\|_{L^{\infty}(\mathcal{GC}_{r_{k}})}^{1+\gamma}\leq\frac{C(k+1)}{9^{k}}\|v\|_{L^{\infty}(\mathcal{GC}_{3R})}^{1+\gamma},
    \end{equation*}
    we can infer from Lemma~\ref{lem. negative eigenvalue equation} a weaker estimate, that in $\mathcal{SC}_{2r_{k+1}}$,
    \begin{equation*}
        v-u\leq(1-\frac{c_{1}}{k+1})\sigma_{k}v
    \end{equation*}
    for exactly the same $c_{1}$ (and also $c_{2}$) given in the proof of Theorem~\ref{thm. limit is 1} (a). By letting
    \begin{equation*}
        \sigma_{k+1}=(1-\frac{c_{1}c_{2}}{k+1})\sigma_{k},
    \end{equation*}
    we can similarly conclude that
    \begin{equation*}
        w_{k+1}=u-(1-\sigma_{k+1})v\geq0\mbox{ in }\mathcal{GC}_{r_{k+1}}.
    \end{equation*}

    This time, however, $\sigma_{k}$ is not a geometric sequence, but instead
    \begin{equation*}
        \sigma_{k}\leq Q\prod_{i=0}^{k-1}(1-\frac{c_{1}c_{2}}{i+1})\sim Q\cdot k^{-c_{1}c_{2}}.
    \end{equation*}
    In other words, there exists some $\epsilon>0$, such that
    \begin{equation*}
        \sigma_{k}\leq Q\cdot(\ln\frac{2R}{r_{k}})^{-\epsilon},\quad\mbox{for all }k\geq0.
    \end{equation*}
    Again, like in Theorem~\ref{thm. limit is 1} (a),
    \begin{equation*}
        \epsilon\sim c_{1}c_{2}\sim\frac{R^{2}}{\|v\|_{L^{\infty}(\mathcal{GC}_{3R})}^{1+\gamma}}.
    \end{equation*}
\end{proof}

\subsection{The Schauder estimate}
In this subsection, let's prove Theorem~\ref{thm. Schauder-Campanato}. In the proof, we will use the constant $\gamma'<\gamma$ mentioned in the Introduction satisfying \eqref{eq. range of gamma'}.

\begin{proof}[Proof of Theorem~\ref{thm. Schauder-Campanato}]
    We consider $r_{k}=\rho^{k}$ for a sufficiently small $\rho$ (to be decided later), and let $h_{k}$ be the harmonic replacement of $u$ in $B_{r_{k}}\cap\Omega$, namely
    \begin{equation*}
        \left\{\begin{aligned}
            &-\Delta h_{k}(X)=0&\mbox{ in }&B_{r_{k}}\cap\Omega\\
            &h_{k}(Y)=u(Y)&\mbox{ on }&\partial(B_{r_{k}}\cap\Omega)
        \end{aligned}\right..
    \end{equation*}
    As $h_{k}\leq u$ wherever they are simultaneously defined, we can also infer that $h_{k+1}\geq h_{k}$ in their common domain. In particular, by "the super-critical interior cone condition" and by \eqref{eq. range of gamma'}, we have that
    \begin{equation*}
        h_{k}(X)\geq h_{0}(X)\geq C^{-1}H(X)\geq C^{-1}|x_{n}-g(x')|^{\frac{2}{1+\gamma'}}
    \end{equation*}
    for a sufficiently large $C$. Upon considering the difference between $u$ and $h_{k}$, we have that $u-h_{k}$ vanishes at $\partial(B_{r_{k}}\cap\Omega)$ and
    \begin{equation*}
        -\Delta(u-h_{k})=f(X)\cdot u^{-\gamma}\leq\Lambda h_{0}^{-\gamma}\leq C|x_{n}-g(x')|^{-\frac{2\gamma}{1+\gamma'}}.
    \end{equation*}
    As $1>\gamma'>\gamma$, we then have $\ds-\frac{2\gamma}{1+\gamma'}>-1$, implying that
    \begin{equation}\label{eq. campanato Ek(X)}
        |u-h_{k}|\leq C r_{k}^{2-\frac{2\gamma}{1+\gamma'}}=:C r_{k}^{\frac{2}{1+\gamma'}+\epsilon_{1}}\quad\mbox{in }B_{r_{k}}\cap\Omega,
    \end{equation}
    see the detailed reason for \eqref{eq. campanato Ek(X)} later in Lemma~\ref{lem. solvability singular} and Remark~\ref{rmk. remark after 5.2} (here, we are using a rescaled version of Lemma~\ref{lem. solvability singular}). Here, we have chosen
    \begin{equation*}
        \epsilon_{1}:=\frac{2(\gamma'-\gamma)}{1+\gamma'}>0.
    \end{equation*}

    Now let's construct the harmonic approximation of $u$. We intend to express
    \begin{equation*}
        h(X)=\mathcal{A}_{\infty}H(X)\mbox{ for some coefficient }\mathcal{A}_{\infty}.
    \end{equation*}
    The idea is to choose a sufficiently small $\epsilon$ (to be decided later), and keep track of two sequences $\mathcal{A}_{k}$ and $\mathcal{B}_{k}$ for $k\geq0$ such that
    \begin{align}\label{eq. campanato iteration assumption k}
        &|u(X)-\mathcal{A}_{k}H(X)|\leq\mathcal{B}_{k}\rho^{(k+1)\epsilon}\Big(\rho^{(k+1)\frac{2}{1+\gamma'}}+H(X)\Big)\\
        =&\mathcal{B}_{k}r_{k+1}^{\epsilon}\Big(r_{k+1}^{\frac{2}{1+\gamma'}}+H(X)\Big)\quad\mbox{in }B_{r_{k+1}}\cap\Omega.
    \end{align}
    Additional, it is fine to separately let $\mathcal{A}_{-1}=0$ and $\mathcal{B}_{-1}=\|u(X)\|_{L^{\infty}(B_{1}\cap\Omega)}$ so that the inequality \eqref{eq. campanato iteration assumption k} automatically holds for $k=-1$ as well.
    
    Now let $k\geq0$ and suppose that $\eqref{eq. campanato iteration assumption k}$ holds for $(k-1)$, namely:
    \begin{align*}
        &|u(X)-\mathcal{A}_{k-1}H(X)|\leq\mathcal{B}_{k-1}\rho^{k\epsilon}\Big(\rho^{k\cdot\frac{2}{1+\gamma'}}+H(X)\Big)\\
        =&\mathcal{B}_{k-1}r_{k}^{\epsilon}\Big(r_{k}^{\frac{2}{1+\gamma'}}+H(X)\Big)\quad\mbox{in }B_{r_{k}}\cap\Omega.
    \end{align*}
    Let's apply the boundary Harnack principle (Lemma~\ref{lem. classical Kemper} or \cite{K72}) to $h_{k}-\mathcal{A}_{k-1}H$ and $H$. Notice that \eqref{eq. H is larger than 2/(1+gamma')} and \cite[Eq (2.5)]{DS20} imply:
    \begin{align*}
        &|h_{k}-\mathcal{A}_{k-1}H|=|u-\mathcal{A}_{k-1}H|\leq\mathcal{B}_{k-1}r_{k}^{\epsilon}\Big(r_{k}^{\frac{2}{1+\gamma'}}+H(X)\Big)\\
        \leq&C\cdot\mathcal{B}_{k-1}r_{k}^{\epsilon}\cdot H(\frac{9}{10}r_{k}\vec{e_{n}})\quad\mbox{on }\partial(B_{r_{k}}\cap\Omega)\subseteq\overline{B_{r_{k}}\cap\Omega},
    \end{align*}
    we can thus choose $\mathcal{A}_{k}$ in a natural way:
    \begin{equation*}
        \mathcal{A}_{k}:=\lim_{X\to0}\frac{h_{k}(X)}{H(X)}=\mathcal{A}_{k-1}+\lim_{X\to0}\frac{h_{k}(X)-\mathcal{A}_{k-1}H(X)}{H(X)}.
    \end{equation*}
    Then by the boundary Harnack principle (Lemma~\ref{lem. classical Kemper} or \cite{K72}) and the assumption \eqref{eq. H is larger than 2/(1+gamma')}, we have
    \begin{equation}\label{eq. campanato 1}
        |\mathcal{A}_{k}-\mathcal{A}_{k-1}|\leq C H(\frac{9}{10}r_{k}\vec{e_{n}})^{-1}\|h_{k}-\mathcal{A}_{k-1}H\|_{L^{\infty}(\partial(B_{r_{k}}\cap\Omega))}\leq Cr_{k}^{\epsilon}\mathcal{B}_{k-1}.
    \end{equation}
    If we further set $\epsilon_{2}$ to be the H\"older exponent obtained by \cite{K72}, such that
    \begin{align*}
        \Big\|\frac{h_{k}(X)-\mathcal{A}_{k-1}H(X)}{H(X)}\Big\|_{C^{\epsilon_{2}}(B_{r_{k+1}}\cap\Omega)}\leq&C r_{k}^{-\epsilon_{2}}\cdot H(\frac{9}{10}r_{k}\vec{e_{n}})^{-1}\|h_{k}-\mathcal{A}_{k-1}H\|_{L^{\infty}(\partial(B_{r_{k}}\cap\Omega))}\\\leq&Cr_{k}^{\epsilon-\epsilon_{2}}\mathcal{B}_{k-1},
    \end{align*}
    then by our choice of $\mathcal{A}_{k}$ such that $\ds\lim_{X\to0}\frac{h_{k}(X)-\mathcal{A}_{k-1}H(X)}{H(X)}=\mathcal{A}_{k}-\mathcal{A}_{k-1}$, we have
    \begin{align}\label{eq. campanato 2}
        &|h_{k}-\mathcal{A}_{k}H|=\Big|\frac{h_{k}(X)-\mathcal{A}_{k-1}H(X)}{H(X)}-(\mathcal{A}_{k}-\mathcal{A}_{k-1})\Big|\cdot H(X)\\
        \leq&C(r_{k}^{\epsilon-\epsilon_{2}}\mathcal{B}_{k-1})\cdot|X|^{\epsilon_{2}}\cdot H(X)\leq C\mathcal{B}_{k-1}r_{k}^{\epsilon-\epsilon_{2}}r_{k+1}^{\epsilon_{2}}H(X)\\
        =&C\mathcal{B}_{k-1}\rho^{k\epsilon+\epsilon_{2}}H(X)\quad\mbox{in }B_{r_{k+1}}\cap\Omega.
    \end{align}
    Taking into account the error term $|u-h_{k}|$ in \eqref{eq. campanato Ek(X)}, we obtain that
    \begin{align*}
        &|u(X)-\mathcal{A}_{k}H(X)|\leq|h_{k}(X)-\mathcal{A}_{k}H(X)|+|u(X)-h_{k}(X)|\\
        \leq&C\mathcal{B}_{k-1}\rho^{k\epsilon+\epsilon_{2}}H(X)+C\rho^{k(\frac{2}{1+\gamma'}+\epsilon_{1})}\\
        \leq&C\Big(\mathcal{B}_{k-1}\rho^{\epsilon_{2}-\epsilon}+\rho^{-\frac{2}{1+\gamma'}-\epsilon+k(\epsilon_{1}-\epsilon)}\Big)\rho^{(k+1)\epsilon}\Big(\rho^{(k+1)\frac{2}{1+\gamma'}}+H(X)\Big)\quad\mbox{in }B_{r_{k+1}}\cap\Omega.
    \end{align*}
    Now let's choose $\epsilon$ and $\rho$ independent of $k$:
    \begin{equation*}
        \epsilon:=\frac{1}{2}\min\{\epsilon_{1},\epsilon_{2}\},\quad\rho:=(2C)^{\frac{1}{\epsilon-\epsilon_{2}}}.
    \end{equation*}
    It would then be natural to set
    \begin{equation}\label{eq. campanato 3}
        \mathcal{B}_{k}:=C\Big(\mathcal{B}_{k-1}\rho^{\epsilon_{2}-\epsilon}+\rho^{-\frac{2}{1+\gamma'}-\epsilon+k(\epsilon_{1}-\epsilon)}\Big)\leq\frac{1}{2}\mathcal{B}_{k-1}+C\rho^{-\frac{2}{1+\gamma'}-\epsilon},
    \end{equation}
    so that \eqref{eq. campanato iteration assumption k} holds for $k$. This completes one loop of the iteration.
    
    Finally, by \eqref{eq. campanato 1}, \eqref{eq. campanato 2} and \eqref{eq. campanato 3}, we see that $\mathcal{B}_{k}$ is uniformly bounded and $\mathcal{A}_{k}$ converges, thus implying the existence of the harmonic approximation for $u(X)$:
    \begin{equation*}
        u(x)\approx h(X)=\mathcal{A}_{\infty}H(X),\quad\mathcal{A}_{\infty}=\lim_{k\to\infty}\mathcal{A}_{k}.
    \end{equation*}
    The positivity of the coefficient $\mathcal{A}_{\infty}$ can be ensured by Theorem~\ref{thm. growth rate} (d).
\end{proof}
The left-over reason for \eqref{eq. campanato Ek(X)} in the proof of Theorem~\ref{thm. Schauder-Campanato} is presented below.
\begin{lemma}\label{lem. solvability singular}
    Let $0<p<1$ be fixed and let $\delta=\delta(p)>0$ be sufficiently small. Assume that $\Omega\subseteq B_{R}$ for a fixed $R$, and assume that there exists some $r>0$, such that for each $Y\in\partial\Omega$, $(\Omega-Y)\cap B_{r}$ is contained in the cone $\{x_{n}\geq-\delta|x'|\}$ up to a rotation. Then, there exists a solution $v$ to
    \begin{equation*}
        \left\{\begin{aligned}
            &-\Delta v(X)=f(X)&\mbox{ in }&\Omega\\
            &v(Y)=0&\mbox{ on }&\partial\Omega
        \end{aligned}\right.,
    \end{equation*}
    where $0\leq f(X)\leq dist(X,\partial\Omega)^{-p}$. Moreover, $\|v\|_{L^{\infty}(\Omega)}\leq C(n,R,\delta,r,p)$.
\end{lemma}
\begin{remark}\label{rmk. remark after 5.2}
    Recall that in Theorem~\ref{thm. limit is continuous when obtuse} (and hence also in Theorem~\ref{thm. Schauder-Campanato}), we have assumed that $\Omega$ is a $C^{1}$ or a convex domain. Then, for such a domain, by choosing $r=r(p,\Omega)$ small in Lemma~\ref{lem. solvability singular}, we have that $(\Omega-Y)\cap B_{r}$ is contained in the cone $\{x_{n}\geq-\delta(p)|x'|\}$ up to a rotation for all $Y\in\partial\Omega$. Therefore, we can apply Lemma~\ref{lem. solvability singular} to get the estimate \eqref{eq. campanato Ek(X)}.
\end{remark}
\begin{proof}[Proof of Lemma~\ref{lem. solvability singular}]
    Denote $\alpha=p^{1/3}\in(0,1)$. We let $w(X)\in L^{\infty}(\overline{\Omega})$ be the solution to
    \begin{equation}\label{eq. w RHS=1}
        \left\{\begin{aligned}
            &-\Delta w(X)=1&\mbox{ in }&\Omega\\
            &w(Y)=0&\mbox{ on }&\partial\Omega
        \end{aligned}\right..
    \end{equation}
    If $\delta=\delta(p)$ is sufficiently small, then the homogeneous harmonic function $H_{\Sigma}$ of the cone $Cone_{\Sigma}=\{x_{n}\geq-\delta|x'|\}$ is of order $\phi_{\Sigma}>\alpha=p^{1/3}$. Then, we have the following estimate:
    \begin{equation*}
        H_{\Sigma}(X)\leq|X|^{\phi_{\Sigma}}\ll|X|^{\alpha}\mbox{ and }|\nabla H_{\Sigma}(X)|\geq c|X|^{\phi_{\Sigma}-1},\quad\mbox{for all }X\in B_{1}\cap Cone_{\Sigma}.
    \end{equation*}
    Let $\overline{\mathcal{H}}(X)=K\cdot H_{\Sigma}(X)^{1-\epsilon}$, where $K>0$ is large and $0<\epsilon<1-\frac{\alpha}{\phi_{\Sigma}}$ is sufficiently small, then it can be verified that $$-\Delta\overline{\mathcal{H}}(X)\geq K\epsilon(1-\epsilon)|X|^{\phi_{\Sigma}(1-\epsilon)-2}\geq1\ \mbox{in} \ B_1\cap Cone_{\Sigma}.$$  Using $\overline{\mathcal{H}}(X+Y)$ (defined in $(B_{r}\cap Cone_{\Sigma})+Y$) as an upper barrier function of \eqref{eq. w RHS=1} at each boundary point $Y\in\partial\Omega$, we see that there exists a constant $C_{1}>0$ depending on $(n,R,r,p)$, such that
    \begin{equation*}
        |w(X)|\leq C_{1}dist(X,\partial\Omega)^{\alpha},\quad\alpha=p^{1/3}.
    \end{equation*}
    We then consider an upper barrier
    \begin{equation*}
        \overline{v}(X)=w(X)^{1-\alpha}.
    \end{equation*}
    It then follows that
    \begin{align*}
        -\Delta\overline{v}(X)=&-(1-\alpha)w(X)^{-\alpha}\Delta w(X)+(1-\alpha)\alpha w(X)^{-\alpha-1}|\nabla w|^{2}\\
        \geq&(1-\alpha)w(X)^{-\alpha}\geq\frac{1-\alpha}{C_{1}^{\alpha}}dist(X,\partial\Omega)^{-\alpha^{2}}\geq C_{2}dist(X,\partial\Omega)^{-p}.
    \end{align*}
    In other words, some multiple of $\overline{v}$ serves as an upper barrier for $v$, and $\overline{v}\Big|_{\partial\Omega}=0$. Noticing that $\underline{v}\equiv0$ is a lower barrier with the same boundary data, we then conclude the existence of $v$ and obtain its $L^{\infty}$ estimate through the Perron method.
\end{proof}

\subsection{Proof of Theorem~\ref{thm. limit is continuous when obtuse}}
When proving Theorem~\ref{thm. limit is continuous when obtuse} in the first situation ($\Gamma\in C^{1}$), we only need to consider the case $0<\gamma<1$. This is because we can directly apply Theorem~\ref{thm. limit is 1} to the case $\gamma>1$, so that $\ds\frac{u}{v}$ extends continuously to $\Gamma$ with limit $1$. Precisely speaking, we easily have:
\begin{corollary}\label{1.7 case 1}
    Assume that $\gamma>1$ and the boundary $\Gamma$ is the graph of a convex or $C^{1}$ function near the origin, then the ratio $u/v$ mentioned in Theorem~\ref{thm. limit is continuous when obtuse} converges to $1$ near the boundary.
\end{corollary}
\begin{proof}
    When $\gamma>1$, we have $\frac{2}{1+\gamma}<1=\phi_{(\partial B_{1})^{+}}$. Then, there exists a sufficiently small $\delta>0$, such that the cone $Cone_{\Sigma}=\{x_{n}\geq-\delta|x'|\}$ satisfies that $\frac{2}{1+\gamma}<\phi_{\Sigma}$. Since $\Gamma\in C^{1}$, we can perform a suitable rotation, such that $0\in\Gamma$, and $|\nabla_{x'}g(x')|\leq\frac{\delta}{2}$ for $|x'|\leq2r$. For every $X=(x',g(x'))\in\Gamma$ with $|x'|\leq r$, we then have the inclusion
    \begin{equation*}
        B_{r}(X)\cap \mathcal{GC}_{3R}\subseteq (X+Cone_{\Sigma})\cap \mathcal{GC}_{3R}.
    \end{equation*}
    Since $\frac{2}{1+\gamma}<\phi_{\Sigma}$, we can apply Theorem~\ref{thm. limit is 1} (a) near each boundary point $X=(x',g(x'))\in\Gamma$ with $|x'|\leq r$. Then, the ratio $\frac{u}{v}$ extends continuously to $\Gamma$ with limit equaling $1$.
\end{proof}

Besides, in the second case ($\Gamma$ is locally convex), if the limiting cone at $X\in B_{r}\cap\Gamma$, defined as
\begin{equation*}
    LC_{X}=\{\vec{v}\neq0:X+\lambda\vec{v}\in B_{2r}\cap\Omega,\mbox{ for some }\lambda>0\}=Cone_{\Sigma_{X}},
\end{equation*}
satisfies that the cone $Cone_{\Sigma_{X}}$ is sub-critical or critical, then similarly we have the continuity of $\ds\frac{u}{v}$ by Theorem~\ref{thm. limit is 1}. In fact, it suffices to verify the additional assumption of Theorem~\ref{thm. limit is 1} (b), that \eqref{eq. growth rate special condition in critical case} has a solution when the cone $Cone_{\Sigma_{X}}$ is critical. To see this, let us assume for simplicity that $X=0$ and write $\Sigma=\Sigma_{X}$.

\begin{lemma}
    Assume that $Cone_{\Sigma}$ is a critical and convex cone (which is also a Lipschitz graph). Let $H_{\Sigma}$ be defined as in \eqref{eq. H Sigma, homogeneous harmonic in a cone}, then the Dirichlet problem \eqref{eq. growth rate special condition in critical case} has a classical solution.
\end{lemma}
\begin{proof}
    The idea is to use the standard Perron's method. A sub-solution of \eqref{eq. growth rate special condition in critical case} is given by $w\equiv0$, which equals the desired boundary data. Then, the main difficulty is to find a suitable super-solution.

    We claim that there exists some $c>0$, such that
    \begin{equation}\label{eq. claim of the growth (lower bound) of harmonic function}
        H_{\Sigma}(X)\geq c\cdot dist(X,\partial Cone_{\Sigma})^{\phi}=c\cdot dist(X,\partial Cone_{\Sigma})^{\frac{2}{1+\gamma}}.
    \end{equation}
    If \eqref{eq. claim of the growth (lower bound) of harmonic function} is true, then we have
    \begin{equation*}
        H_{\Sigma}(X)^{-\gamma}\leq C\cdot dist(X,\partial Cone_{\Sigma})^{-\phi\gamma},\quad\mbox{with }\phi\gamma=\frac{2\gamma}{1+\gamma}<2.
    \end{equation*}
    Then, we choose $K$ large and construct $\overline{w}$ in the following way:
    \begin{equation*}
        \overline{w}=K\cdot d(X)^{\epsilon},\quad d(X)=dist\Big(X,\partial(B_{1}\cap Cone_{\Sigma})\Big),\quad\epsilon=\frac{1}{1+\gamma}.
    \end{equation*}
    We see that $\overline{w}$ is continuous and vanishes on the boundary $\partial(B_{1}\cap Cone_{\Sigma})$. Besides,
    \begin{equation*}
        \Delta\overline{w}=\epsilon K\cdot d(X)^{\epsilon-1}\Delta d(X)-(\epsilon-\epsilon^{2})K\cdot d(X)^{\epsilon-2}|\nabla d(X)|^{2}.
    \end{equation*}
    Since the region $B_{1}\cap Cone_{\Sigma}$ is convex, we see that $d(X)$ is super-harmonic and $|\nabla d(X)|=1$ almost everywhere near $\partial(B_{1}\cap Cone_{\Sigma})$. This implies:
    \begin{equation*}
        -\Delta\overline{w}\geq(\epsilon-\epsilon^{2})K\cdot d(X)^{\epsilon-2}\geq C\cdot d(X)^{\epsilon-2}\geq C\cdot dist(X,\partial Cone_{\Sigma})^{-\phi\gamma}.
    \end{equation*}
    In conclusion, $\overline{w}$ is the desired super-solution to \eqref{eq. growth rate special condition in critical case} with vanishing boundary data.

    We finally prove the claim \eqref{eq. claim of the growth (lower bound) of harmonic function}. We write $Cone_{\Sigma}$ as a Lipschitz graph $\{x_{n}\geq g(x')\}$. For each $X=(x',x_{n})\in Cone_{\Sigma}$, we define its height $h$ and its "foot" $X^{foot}$ as follows:
    \begin{equation*}
        h=h(X):=x_{n}-g(x')\geq0,\quad X^{foot}:=(x',g(x')).
    \end{equation*}
    Since $Cone_{\Sigma}$ is a Lipschitz graph, we have $dist(X,\partial Cone_{\Sigma})\sim h$. Then, \eqref{eq. claim of the growth (lower bound) of harmonic function} is equivalent to $H_{\Sigma}(X)\gtrsim h(X)^{\phi}$. To see this, we consider two cases: $h(X)\leq\frac{1}{10}|X|$, and $h(X)\geq\frac{1}{10}|X|$.
    
    \textbf{Case 1: $h(X)\geq\frac{1}{10}|X|$.} In this case, we let
        \begin{equation*}
            Y=|X|\vec{e_{n}}\in Cone_{\Sigma},
        \end{equation*}
        then we have $h(Y)=|X|\leq10h(X)$. Moreover, since $Cone_{\Sigma}$ is a Lipschitz graph, we have $|X|\geq\frac{1}{\sqrt{L^{2}+1}}h(X)$, where $L=[g(x')]_{C^{0,1}(\mathbb{R}^{n-1})}<\infty$ is the Lipschitz semi-norm of $\partial Cone_{\Sigma}$. By the fact that $h(Y)\sim dist(Y,\partial Cone_{\Sigma})$, we have:
        \begin{equation*}
            dist(X,\partial Cone_{\Sigma})\sim dist(Y,\partial Cone_{\Sigma})\gtrsim dist(X,Y).
        \end{equation*}
        We then apply the interior Harnack principle to $H_{\Sigma}$, and get that
        \begin{equation*}
            H_{\Sigma}(X)\geq c H_{\Sigma}(Y)\geq c h(Y)^{\phi}\geq c\cdot dist(X,\partial Cone_{\Sigma})^{\phi}.
        \end{equation*}
        Here, the estimate $H_{\Sigma}(Y)\gtrsim h(Y)^{\phi}$ follows from the "frequency" assumption of $Cone_{\Sigma}$ and that $Y$ is on the $\vec{e_{n}}$-axis. Therefore, \eqref{eq. claim of the growth (lower bound) of harmonic function} holds in Case 1.
        
        \textbf{Case 2: $h(X)\leq\frac{1}{10}|X|$.} By the convexity of $Cone_{\Sigma}$, we have the geometric fact:
        \begin{equation*}
            Cone_{\Sigma}+X\subseteq Cone_{\Sigma}.
        \end{equation*}
        Let us consider the point $Y$ (different from the one in Case 1) defined as:
        \begin{equation*}
            Y=X^{foot}+10\sqrt{L^{2}+1}|x'|\vec{e_{n}}\in Cone_{\Sigma}+X.
        \end{equation*}
        Then, $Y\in Cone_{\Sigma}$ and $h(Y)=10\sqrt{L^{2}+1}|x'|\geq\frac{1}{10}|Y|$. It follows from Case 1 that
        \begin{equation*}
            H_{\Sigma}(Y)\geq c\cdot h(Y)^{\phi}.
        \end{equation*}
        Let $\eta\geq0$ be a smooth cut-off function supported in the spherical cap $\Sigma$, such that $\eta\equiv1$ near the north pole $\vec{e_{n}}$. Consider the solution $\beta(Z)$ to the following Dirichlet problem of the Laplace equation:
        \begin{equation*}
            \left\{\begin{aligned}
                &\Delta\beta(Z)=0&\mbox{in }&X^{foot}+(B_{h(Y)}\cap Cone_{\Sigma})\\
                &\beta(Z)=0&\mbox{on }&X^{foot}+(B_{h(Y)}\cap\partial Cone_{\Sigma})\\
                &\beta(Z)=H_{\Sigma}(Z)\cdot\eta\Big(\frac{Z-X^{foot}}{h(Y)}\Big)&\mbox{on }&X^{foot}+(Cone_{\Sigma}\cap\partial B_{h(Y)})
            \end{aligned}\right..
        \end{equation*}
        In fact, the existence of $\beta(Z)$ follows from the Lipschitz regularity of $B_{h(Y)}\cap Cone_{\Sigma}$ and by the continuity of the boundary condition. As the boundary data of $\beta(Z)$ is less than $H_{\Sigma}(Z)$, we must have $H_{\Sigma}(Z)\geq\beta(Z)$ in $X^{foot}+(B_{h(Y)}\cap Cone_{\Sigma})$ by the maximal principle.
        
        For $Z\in X^{foot}+(Cone_{\Sigma}\cap\partial B_{h(Y)})$ with $\eta\Big(\frac{Z-X^{foot}}{h(Y)}\Big)=1$, we have
        \begin{equation*}
            \beta(Z)=H_{\Sigma}(Z)\geq c\cdot H_{\Sigma}(Y)\geq c\cdot h(Y)^{\phi}
        \end{equation*}
        by applying the interior Harnack principle to $H_{\Sigma}$ near $Y$. Using a positive multiple of the following sub-harmonic lower barrier (where $\epsilon$ depends only on $\Sigma$ and $\eta$):
        \begin{equation*}
            \widetilde{\beta}(Z)=|Z-(1+\epsilon^{2})Y|^{1.9-n}-(\epsilon|Y|)^{1.9-n},\quad Z\in B_{\epsilon|Y|}\Big((1+\epsilon^{2})Y\Big),
        \end{equation*}
        we conclude that $\beta\Big((1-\epsilon^{2})Y\Big)\geq c\cdot h(Y)^{\phi}$. Then, we apply the boundary Harnack principle (Lemma~\ref{lem. classical Kemper} or \cite{K72}) to $\beta(Z)$ and $H_{\Sigma}(Z-X^{foot})$, and have that:
        \begin{equation*}
            \frac{\beta(X)}{\beta\Big((1-\epsilon^{2})Y\Big)}\geq c\frac{H_{\Sigma}(X-X^{foot})}{H_{\Sigma}\Big((1-\epsilon^{2})Y-X^{foot}\Big)}=c\cdot\Big(\frac{h(X)}{(1-\epsilon^{2})h(Y)}\Big)^{\phi}.
        \end{equation*}
        This gives the lower bound $\beta(X)\geq c\cdot h(X)^{\phi}$. Recall that $\beta(Z)\leq H_{\Sigma}(Z)$ inside $X^{foot}+(B_{h(Y)}\cap Cone_{\Sigma})$, which contains the point $X$, we have then proven \eqref{eq. claim of the growth (lower bound) of harmonic function} in Case 2.
\end{proof}

We summarize the discussions above and state the following fact:
\begin{corollary}\label{1.7 case 2}
    Assume that the boundary $\Gamma$ is the graph of a convex function near the origin. If the limiting cone at the origin is $Cone_{\Sigma}$, such that the cone $Cone_{\Sigma}$ is sub-critical or critical, then the ratio $u/v$ mentioned in Theorem~\ref{thm. limit is continuous when obtuse} converges to $1$ when approaching the origin.
\end{corollary}

In all other situations, we have the following key observation, that "the super-critical interior cone condition" \eqref{eq. super-critical interior cone condition} holds for some cone $Cone_{\Sigma}$, see Definition~\ref{def. super-critical interior cone condition}.

In fact, when $\Gamma$ is $C^{1}$, we can set $Cone_{\Sigma}=\{x_{n}\geq\delta|x'|\}$ and
\begin{equation*}
    \gamma'=\frac{1+\gamma}{2}<1.
\end{equation*}
Given that $\delta>0$ is sufficiently small, then \eqref{eq. range of gamma'} holds, and $Cone_{\Sigma}$ is the interior cone of every boundary point near the origin.

When $\Gamma$ is a (locally) convex graph, then $\Omega\cap B_{2r}$ is a convex domain for a sufficiently small $r$. Let $LC_{0}=Cone_{\Sigma_{0}}$ be the limiting cone at the origin such that $Cone_{\Sigma_{0}}$ is super-critical (the case that $Cone_{\Sigma_{0}}$ is sub-critical or critical is previously discussed). Let $\phi_{0}=\phi_{\Sigma_{0}}$ be the "frequency" of the limit cone $LC_{0}$, then $\phi_{0}=\frac{2}{1+\gamma_{0}}$ for some $\gamma_{0}>\gamma$. We then let
\begin{equation*}
    \gamma'=\frac{\gamma_{0}+\gamma}{2},
\end{equation*}
meaning $\gamma_{0}>\gamma'>\gamma$. Recall that as $g$ is (locally) convex,
\begin{equation*}
    \liminf_{X\in\Gamma,X\to0}B_{1}\cap LC_{X}\supseteq B_{1}\cap LC_{0}
\end{equation*}
in the Hausdorff sense. In other words, there exist $\epsilon>0$ and $\Sigma\subseteq\Sigma_{0}\subseteq\partial B_{1}$ such that
\begin{itemize}
    \item[(1)] $\phi_{\Sigma}<\frac{2}{1+\gamma'}$, or in other words, $Cone_{\Sigma}$ is a super-critical cone with respect to $\gamma'$;
    \item[(2)] For every $X\in B_{\epsilon}\cap\Gamma$, we have $X+(B_{\epsilon}\cap Cone_{\Sigma})\subseteq B_{r}\cap\Omega$.
\end{itemize}

After all, we have shown why we have "the super-critical interior cone condition" \eqref{eq. super-critical interior cone condition} when $\Omega$ has a $C^{1}$ or a convex boundary near the origin. Moreover, notice that when $\Gamma$ is a $C^{1}$ or a convex boundary, there is no essential distinction between $\mathcal{GC}_{r}$ and $B_{r}\cap\Omega$, so the open sets in Theorem~\ref{thm. Schauder-Campanato} can be written as $B_{r}\cap\Omega$.

Now we can prove Theorem~\ref{thm. limit is continuous when obtuse} under "the super-critical interior cone condition" \eqref{eq. super-critical interior cone condition} and \eqref{eq. range of gamma'}. Notice that all other cases in Theorem~\ref{thm. limit is continuous when obtuse} were already discussed in Corollary~\ref{1.7 case 1} and Corollary~\ref{1.7 case 2}.
\begin{proof}[Proof of Theorem~\ref{thm. limit is continuous when obtuse}, under "the super-critical interior cone condition"]
    We let $X\in\mathcal{GC}_{r}$ such that
    \begin{equation*}
        X=(x',g(x')+t).
    \end{equation*}
    If $X$ tends to the origin, then $x',t\to0$. We intend to show that
    \begin{equation*}
        \lim_{X\to0}\frac{u(X)}{H(X)}=\lim_{x',t\to0}\frac{u(X)}{H(X)}=\mathcal{A}_{0}
    \end{equation*}
    for some non-zero $\mathcal{A}_{0}$. In fact, from Theorem~\ref{thm. Schauder-Campanato}, there exists $\mathcal{A}_{0}>0$ and $\mathcal{A}_{X^{foot}}$ for each
    \begin{equation*}
        X^{foot}:=(x',g(x'))
    \end{equation*}
    near the origin, such that it holds near the origin that
    \begin{align*}
        |\frac{u(Z)}{H(Z)}-\mathcal{A}_{0}|\leq&C\frac{|Z|^{\frac{2}{1+\gamma'}+\epsilon}}{H(Z)}+C|Z|^{\epsilon},\\
        |\frac{u(Z)}{H(Z)}-\mathcal{A}_{X^{foot}}|\leq&C\frac{|Z-X^{foot}|^{\frac{2}{1+\gamma'}+\epsilon}}{H(Z)}+C|Z-X^{foot}|^{\epsilon}.
    \end{align*}
    Now we choose
    \begin{equation*}
        Z=X^{foot}+|x'|\vec{e_{n}}
    \end{equation*}
    in both inequalities, then
    \begin{equation*}
        |\frac{u(Z)}{H(Z)}-\mathcal{A}_{0}|\leq C(L+2)^{\frac{2}{1+\gamma'}+\epsilon}\frac{|x'|^{\frac{2}{1+\gamma'}+\epsilon}}{H(Z)}+C(L+2)^{\epsilon}|x'|^{\epsilon}
    \end{equation*}
    and
    \begin{equation*}
        |\frac{u(Z)}{H(Z)}-\mathcal{A}_{X^{foot}}|\leq C\frac{|x'|^{\frac{2}{1+\gamma'}+\epsilon}}{H(Z)}+C|x'|^{\epsilon}.
    \end{equation*}
    Combining these above yields that
    \begin{equation*}
        |\mathcal{A}_{X^{foot}}-\mathcal{A}_{0}|\leq C\frac{|x'|^{\frac{2}{1+\gamma'}+\epsilon}}{H(Z)}+C|x'|^{\epsilon}\leq C|x'|^{\epsilon},
    \end{equation*}
    where we have used the lower bound obtained from \eqref{eq. super-critical interior cone condition} and \eqref{eq. range of gamma'}:
    \begin{equation*}
        H(Z)\geq c|Z-X^{foot}|^{\frac{2}{1+\gamma'}}=c|x'|^{\frac{2}{1+\gamma'}}.
    \end{equation*}
    Therefore, as $X=(x',g(x')+t)\to0$,
    \begin{equation*}
        |\frac{u(X)}{H(X)}-\mathcal{A}_{0}|\leq|\frac{u(X)}{H(X)}-\mathcal{A}_{X^{foot}}|+|\mathcal{A}_{X^{foot}}-\mathcal{A}_{0}|\leq Ct^{\epsilon}+C|x'|^{\epsilon}\to0.
    \end{equation*}
    
    The argument also works for the other function $v(X)$, so
    \begin{equation*}
        \lim_{X\to0}|\frac{v(X)}{H(X)}-\mathcal{A}_{0}'|=0\quad\mbox{for some }\mathcal{A}_{0}'>0.
    \end{equation*}
    Therefore, $u/v$ is continuous near the boundary, so Theorem~\ref{thm. limit is continuous when obtuse} itself is proven.
\end{proof}

\section{Examples}\label{sec. examples}
In this section, we let $f(X)\equiv1$, so that $u(X)$ satisfies $-\Delta u=u^{-\gamma}$. We hope that the readers can understand the main results better with the help of several examples. After that, we give a proof of Theorem~\ref{thm. ratio not continuous}, incorporating the examples presented below.
\subsection{Smooth boundaries}
We first consider a domain with smooth boundaries. Obviously, the simplest smooth boundaries are flat or spherical. The readers can pay attention to how the main results fit into the present example.
\begin{example}[half space]\label{ex. straight boundary}
    We consider $g(x')=0$, so that the boundary $\Gamma$ is (locally) flat. We would like to consider one-dimensional positive solutions $u(X)=u(x_{n})$ defined above $\Gamma$.

    We then have that
    \begin{equation*}
        u(0)=0\mbox{ and }u''(t)=-u(t)^{-\gamma}.
    \end{equation*}
    We multiply by $u'(t)$ on both sides of the ODE, and then have that
    \begin{align*}
        \frac{d}{dt}\{u'(t)\}^{2}=&2u''(t)u'(t)=-2u(t)^{-\gamma}u'(t)\\
        =&\left\{\begin{aligned}
            &\frac{d}{dt}\Big\{\frac{2}{\gamma-1}u(t)^{1-\gamma}\Big\}&,&\mbox{ if }\gamma\neq1\\
            &\frac{d}{dt}\Big\{-2\ln{u(t)}\Big\}&,&\mbox{ if }\gamma=1
        \end{aligned}\right..
    \end{align*}
    
    If $\gamma<1$, then near the origin we have
    \begin{equation*}
        u'(t)=\sqrt{K^{2}-\frac{2}{1-\gamma}u(t)^{1-\gamma}}.
    \end{equation*}
    At $t=0$, we have $u(0)=0$, so $u'(0)=K$. Besides, we can deduce that $u$ cannot exceed the value $\ds\Big(\frac{1-\gamma}{2}K^{2}\Big)^{\frac{1}{1-\gamma}}$. After an integration we have
    \begin{equation}\label{eq. inverse of u}
        F(u(T)):=\int_{0}^{u(T)}\frac{ds}{\sqrt{K^{2}-\frac{2}{1-\gamma}s^{1-\gamma}}}=\int_{0}^{T}\frac{u'(t)dt}{\sqrt{K^{2}-\frac{2}{1-\gamma}u(t)^{1-\gamma}}}=T.
    \end{equation}
    This gives us the way to "explicitly" solve the ODE (if we know how to integrate).
    
    When $u(T^{*})$ reaches the maxima $\ds\Big(\frac{1-\gamma}{2}K^{2}\Big)^{\frac{1}{1-\gamma}}$, we just set
    \begin{equation*}
        u(t)=u(2T^{*}-t)\mbox{ for }t\in[T^{*},2T^{*}].
    \end{equation*}
    In particular, $u(2T^{*})=0$ and $u(t)$ cannot be defined beyond $2T^{*}$. Moreover we can even express $T^{*}$ using the integration:
    \begin{equation*}
        T^{*}=\int_{0}^{\Big(\frac{1-\gamma}{2}K^{2}\Big)^{\frac{1}{1-\gamma}}}\frac{ds}{\sqrt{K^{2}-\frac{2}{1-\gamma}s^{1-\gamma}}}.
    \end{equation*}
    As $K\to+\infty$, we see $T^{*}\to\infty$. Therefore we can choose $K$ large enough so that $u$ is solvable in $[0,100]$ for example. In fact, we see for different $K$, the solutions $u$ are "similar to each other" via the invariant scaling $\widetilde{u}(t)=R^{\frac{-2}{1+\gamma}}u(Rt)$.

    When $\gamma>1$, then near the origin we have
    \begin{equation*}
        u'(t)=\sqrt{\frac{2}{\gamma-1}u(t)^{1-\gamma}+C}.
    \end{equation*}
    Similar computation of integration will explicitly express the solutions and we omit the details. It could be seen that near the origin we have the expansion
    \begin{equation*}
        u(t)=\Big(\frac{(1+\gamma)^{2}}{2\gamma-2}\Big)^{\frac{1}{1+\gamma}}t^{\frac{2}{1+\gamma}}+o(t^{\frac{2}{1+\gamma}}).
    \end{equation*}
    
    We would like to stress that in the case $\gamma>1$, the value of $C$ affects the asymptotic behavior of $u$ for large $t$. When $C<0$, then $u$ can only be defined in a finite interval. When $C=0$, then $u$ is defined in $[0,+\infty)$ and is a power function of degree $\frac{2}{1+\gamma}$. When $C>0$, then $u$ is also defined in $[0,+\infty)$, and it is of linear growth as $t\to\infty$ with the asymptotic slope $\sqrt{C}$.

    When $\gamma=1$, then near the origin, we have
    \begin{equation*}
        u'(t)=\sqrt{2\ln{\frac{1}{u(t)}}+C}.
    \end{equation*}
    Similar to the case $\gamma<1$, by noticing that $2\ln{\frac{1}{s}}+C\sim(e^{C/2}-s)$ as $s\to e^{C/2}-$, we see
    \begin{equation*}
        \int_{0}^{e^{C/2}}\frac{ds}{\sqrt{2\ln{\frac{1}{s}}+C}}=:T^{*}<\infty.
    \end{equation*}
    Similar to \eqref{eq. inverse of u}, we have that $u(t)$ is the inverse function of $G(v):=\ds\int_{0}^{v}\frac{ds}{\sqrt{2\ln{\frac{1}{s}}+C}}$ for $t\in[0,T^{*}]$ with $u(T^{*})=e^{C/2}$ and $u'(T^{*})=0$. For $t>T^{*}$,  it satisfies the first order ODE:
    \begin{equation*}
        u'(t)=-\sqrt{2\ln{\frac{1}{u(t)}}+C},\quad\mbox{for } t>T^{*}.
    \end{equation*}
    Then one can verify  that the even extended function $u(t):=u(2T^{*}-t)$ for $t\in[T^*,2T^*]$ solves the ODE $u''=-u^{-1}$. It is defined and supported in $[0,2T^{*}]$.

    The growth rate of $u$ near the origin is of order $t\sqrt{\ln{\frac{1}{t}}}$. It could be obtained from Theorem~\ref{thm. growth rate} (e) and Theorem~\ref{thm. growth rate, (b) improved}, but it could also be obtained from the ODE. In fact, by direct computation we have
    \begin{equation*}
        (t\sqrt{\ln{\frac{1}{t}}})''=-(\frac{1}{2}+\frac{1}{4\ln{\frac{1}{t}}})\frac{1}{t\sqrt{\ln{\frac{1}{t}}}}.
    \end{equation*}
    Therefore, $At\sqrt{\ln{\frac{1}{t}}}$ is a sub-solution to the ODE when $A$ is sufficiently small, and is a super-solution when $A$ is sufficiently large.
\end{example}
Now we consider a rotationally symmetric solution defined in an annulus.
\begin{example}[exterior ball]\label{ex. spherical boundary}
    Let $A_{R,r}=B_{R}\setminus B_{r}$ be an annulus. We require $0<r<R<\infty$. If we impose a rotationally symmetric boundary condition, then by the maximal principle of "SLEF", the solution $u(X)$ must be also rotationally symmetric. Let $t=|X|-r$, then the ODE satisfied by $u(X)=u(t)$ satisfies
    \begin{equation*}
        u''(t)+\frac{n-1}{r+t}u'(t)=-u(t)^{-\gamma}.
    \end{equation*}
    Let's assume that the boundary condition on $\partial B_{r}$ is zero. By Theorem~\ref{thm. growth rate}, we see that $u(t)$ is of growth order $\frac{2}{1+\gamma}$ when $\gamma>1$, while $u(t)$ is of growth order $1$ when $\gamma<1$.
    
    We should notice that when $\gamma<1$, $u'(t)$ is bounded near the origin, and $u(t)^{-\gamma}\lesssim t^{-\gamma}$ by Theorem~\ref{thm. growth rate} (d). Therefore, $u''(t)\sim t^{-\gamma}$ near $t=0$, so $u(t)$ is in fact $C^{1,1-\gamma}$ near $t=0$.
\end{example}
\subsection{Corner regions}
We now consider a two dimensional corner, which could be written as $\{x_{2}\geq k|x_{1}|\}$. The angle of such a corner will then be $\theta=\pi-2\tan^{-1}{k}$.
\begin{example}\label{ex. corner phi value}
    In \cite[Lemma 3.5]{EH23}, Elgindi and Huang constructed a homogeneous solution to $-\Delta u=u^{-\gamma}$ in the first quadrant of $\mathbb{R}^{2}$, in which case $\theta=\frac{\pi}{2}$. When $0<\gamma<1$, they considered an angular ODE in the interval $[0,\frac{\pi}{2}]$, and proved its solvability.

    Based on some discussion with Huang, we realize that for a generic corner with angle $\theta$, the homogeneous solution to $-\Delta u=u^{-\gamma}$ exists if and only if $\theta<\frac{1+\gamma}{2}\pi$. The proof will be the same as that in \cite{EH23}. This coincides with Theorem~\ref{thm. growth rate}.

    In fact, let $\Sigma\in\partial B_{1}$ be an arc with length $\theta$, so that the corner is equal to $Cone_{\Sigma}$, then the cone $Cone_{\Sigma}$ is
    \begin{itemize}
        \item sub-critical, when $\theta<\frac{1+\gamma}{2}\pi$;
        \item critical, when $\theta=\frac{1+\gamma}{2}\pi$;
        \item super-critical, when $\theta>\frac{1+\gamma}{2}\pi$.
    \end{itemize}
\end{example}
Let's then consider a boundary Harnack principle in the first quadrant of $\mathbb{R}^{2}$.
\begin{example}
    In \cite[Theorem B.1]{HZ24}, Huang and the third author of the present paper proved a boundary Harnack principle in the first quadrant
    \begin{equation*}
        \mathbb{R}^{++}:=\{(x,y)\in\mathbb{R}^{2}:x,y>0\}.
    \end{equation*}
    That if $u$ and the homogeneous function $\Psi$ solve the "SLEF" in the first quadrant, and they both vanish at the boundary, then
    \begin{equation*}
        \frac{u}{\Psi}-1=O(|X|^{\epsilon})
    \end{equation*}
    near the origin. (In \cite{EH23}, such a ratio is implicitly known to be $L^{p}$ for an arbitrarily large $p$.) Notice that in \cite{HZ24}, the H\"older exponent $\epsilon$ does not depend on $\|u\|_{L^{\infty}}$, which is better than Theorem~\ref{thm. limit is 1} of the present paper.
\end{example}
Finally for the case of a corner region, we would like to ask a question regarding the optimality issue in Theorem~\ref{thm. growth rate} (b)(e) and Theorem~\ref{thm. growth rate, (b) improved}. It is interesting to see if Theorem~\ref{thm. growth rate, (b) improved} still holds when we remove the solvability condition \eqref{eq. growth rate special condition in critical case}. The authors tend to believe that the answer is "No". Precisely, we consider an angle in $\mathbb{R}^{2}$ and make the following conjecture:

Let $\ds\theta=\frac{1+\gamma}{2}\pi$, and $Cone_{\Sigma}$ be an angle in $\mathbb{R}^{2}$ with  opening $\theta$. We conjecture that the optimal growth rate estimate given in Theorem~\ref{thm. growth rate, (b) improved} holds if and only if $\gamma<2$ (thus $\ds\theta<\frac{3\pi}{2}$).

\subsection{Failure of the ratio to be continuous}
In this subsection, we prove Theorem~\ref{thm. ratio not continuous}. We consider the curve $x_{2}=g(x_{1})$ in $\mathbb{R}^{2}$, where $g(x_{1}):[-1,1]\to\mathbb{R}$ is a piecewise function given by
\begin{equation}\label{eq. strange domain}
    g(x_{1})=\left\{\begin{aligned}
        &|x_{1}-\frac{1}{i}|-R+\sqrt{R^{2}-\frac{1}{i^{2}}}&,&\mbox{ if }|x_{1}-\frac{1}{i}|\leq R-\sqrt{R^{2}-\frac{1}{i^{2}}}\\
        &0&,&\mbox{ otherwise}
    \end{aligned}\right..
\end{equation}
Here,
\begin{equation*}
    i\in\mathbb{Z}\setminus\{0\}=\{\pm1,\pm2,\pm3,\cdots\}.
\end{equation*}
The radius $R$ is large, so that all intervals $\ds\Big\{|x_{1}-\frac{1}{i}|\leq R-\sqrt{R^{2}-\frac{1}{i^{2}}}\Big\}$ do not overlap.

Let's roughly describe the geometry of such a curve. The curve is bounded between a straight line $\mathcal{C}_{1}=\{x_{2}=0\}$ and a circular arc $\mathcal{C}_{2}=\{x_{2}=-R+\sqrt{R^{2}-x_{1}^{2}}\}$, which is a part of the circle $\{x_{1}^{2}+(x_{2}+R)^{2}=R^{2}\}$. The two curves $\mathcal{C}_{1}$ and $\mathcal{C}_{2}$ are tangent to each other at the origin. Besides, we realize that at each point $(i^{-1},g(i^{-1}))$, the curve $\{x_{2}=g(x_{1})\}$ has a $\frac{\pi}{2}$ angle (see Figure \ref{fig-g}).

\begin{figure}
    \centering
\includegraphics[width=1\linewidth]{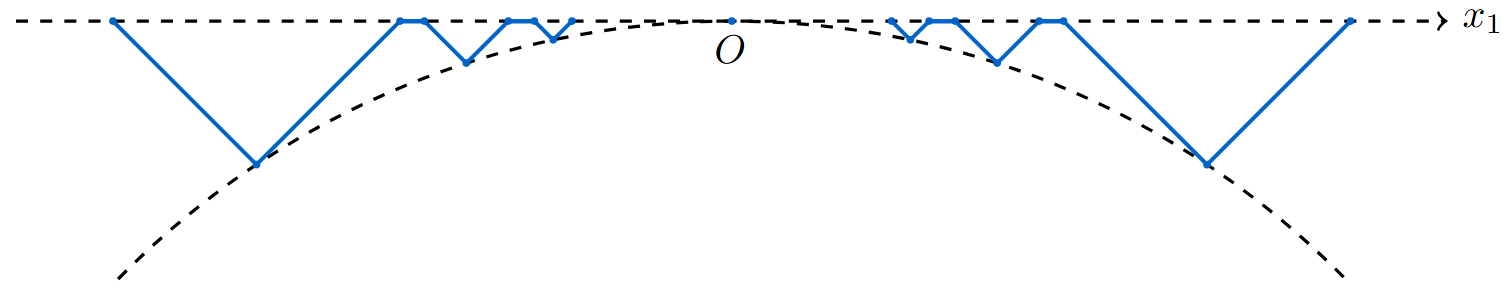}
    \caption{Graph of $g$.}
    \label{fig-g}
\end{figure}

We will fix a $\gamma\in(0,1)$. The goal in this subsection is to show that above such a curve
\begin{equation*}
    \Gamma=\{x_{2}=g(x_{1})\},
\end{equation*}
there are two solutions $u$, $v$ to the "SLEF" such that they both vanish at $\Gamma$ but their ratio $u/v$ is not continuous at the origin.

The key trick is the following. As $\gamma>0$, the limiting cone at $X=\ds(i^{-1},g(i^{-1}))$ is an angle of size $\frac{\pi}{2}$, which means that $Cone_{\Sigma}$ is sub-critical (recall Example~\ref{ex. corner phi value}) for
\begin{equation*}
    \Sigma=\{(\cos{\theta},\sin{\theta}):\frac{\pi}{4}<\theta<\frac{3\pi}{4}\}.
\end{equation*}
Then no matter how $u$ and $v$ are constructed, we always have (using Theorem~\ref{thm. limit is 1} or \cite[Theorem B.1]{HZ24}) that
\begin{equation*}
    \lim_{X\to(i^{-1},g(i^{-1}))}\frac{u(X)}{v(X)}=1.
\end{equation*}
Since $(i^{-1},g(i^{-1}))$ tends to the origin as $i\to\pm\infty$, we should expect $\frac{u}{v}$ to be $1$ at the origin, if it were continuous. Therefore, our strategy is to show that
\begin{equation*}
    \limsup_{X=(0,t),\ t\to0}\frac{u(X)}{v(X)}\geq2
\end{equation*}
for some intentionally chosen $u$ and $v$.
\begin{proof}[Proof of Theorem~\ref{thm. ratio not continuous}]
    From the discussion above, the idea is then to introduce two auxiliary functions $\varphi$ and $\psi$, so that they both vanish at the origin and
    \begin{equation*}
        \frac{\partial_{x_{2}}\varphi(0)}{\partial_{x_{2}}\psi(0)}\geq2.
    \end{equation*}
    We then use $\varphi$ and $\psi$ as the boundary value to construct two solutions $u\geq\varphi$ and $v\leq\psi$ to the "SLEF". Then $\ds\frac{u}{v}$ is not continuous at the origin.

    To ensure $u\geq\varphi$, we let $\varphi(X)=\varphi(x_{2})$ be a continuous function defined near $0$, such that
    \begin{align*}
        &-\varphi(x_{2})''=\varphi(x_{2})^{-\gamma}\mbox{ for }x_{2}\geq0,\\
        &\varphi'(0)=2k\mbox{ as the right-side derivate},
    \end{align*}
    and $\varphi(x_{2})=0$ for $x_{2}\leq0$. We then let $u(X)$ satisfy
    \begin{equation*}
        \left\{\begin{aligned}
            &-\Delta u(X)=u(X)^{-\gamma}&\mbox{ in }&\{|x_{1}|\leq1,g(x_{1})\leq x_{2}\leq1\}\\
            &u(Y)=\varphi(Y)&\mbox{ on }&\partial\{|x_{1}|\leq1,g(x_{1})\leq x_{2}\leq1\}
        \end{aligned}\right..
    \end{equation*}
    As $u>0=\varphi$ on the line $\{x_{2}=0\}$, we see
    \begin{equation*}
        u\geq\varphi\mbox{ in }\{|x_{1}|\leq1,g(x_{1})\leq x_{2}\leq1\}.
    \end{equation*}

    Similarly, to ensure $v\leq\psi$, we let $\psi(X)=\psi(d)$ where $d$ is the signed distance to the circle $\{x_{1}^{2}+(x_{2}+R)^{2}=R^{2}\}$. We let $\psi(0)=0$, $\psi'(0)=k$ and
    \begin{equation*}
        \psi''(d)+\frac{1}{R+d}\psi'(d)=-\psi(d)^{-\gamma}.
    \end{equation*}
    Notice that we have
    \begin{equation*}
        \ds\frac{\partial_{x_{2}}\varphi(0)}{\partial_{x_{2}}\psi(0)}=\frac{2k}{k}=2.
    \end{equation*}
    Moreover, Example~\ref{ex. straight boundary} and Example~\ref{ex. spherical boundary} guarantee the existence of $\varphi$ and $\psi$, as well as the slope $k>0$.
    
    We now let $v(X)$ satisfy
    \begin{equation*}
        \left\{\begin{aligned}
            &-\Delta v(X)=v(X)^{-\gamma}&\mbox{ in }&\{|x_{1}|\leq1,g(x_{1})\leq x_{2}\leq1\}\\
            &v(Y)=0&\mbox{ on }&\Gamma=\{x_{2}=g(x_{1})\}\\
            &v(Y)=\psi(Y)&\mbox{ on }&\partial\{|x_{1}|\leq1,g(x_{1})\leq x_{2}\leq1\}\setminus\Gamma
        \end{aligned}\right..
    \end{equation*}
    As the boundary value is continuous on $\partial\{|x_{1}|\leq1,g(x_{1})\leq x_{2}\leq1\}$, $v(X)$ must exist. As $v=0\leq\psi$ on $\Gamma$, we see
    \begin{equation*}
        v\leq\psi\mbox{ in }\{|x_{1}|\leq1,g(x_{1})\leq x_{2}\leq1\}.
    \end{equation*}
\end{proof}
\begin{remark}
    One should notice that $x_{2}=g(x_{1})$ constructed above is not a convex graph. It is exactly the example above that inspired the authors to study the continuity of $\frac{u}{v}$ when $\Gamma$ has a convex boundary, as was presented in Theorem~\ref{thm. limit is continuous when obtuse}.
\end{remark}
Finally, we would like to mention that despite its strange shape, our construction of the domain in \eqref{eq. strange domain} can greatly simplify our computation when proving Theorem~\ref{thm. ratio not continuous}. We would like to ask if there is an alternative proof of Theorem~\ref{thm. ratio not continuous} when the boundary $\Gamma$ is simpler. For example, when $u$ and $v$ satisfy
\begin{equation*}
    -\Delta u=u^{-\gamma},\ -\Delta v=v^{-\gamma},\ \mbox{with }1<\gamma<2,
\end{equation*}
in $\Omega=\{x_{2}>-|x_{1}|\}\cap B_{1}$, and both vanish at $\Gamma=\{x_{2}=-|x_{1}|\}$. We guess that the ratio $\frac{u}{v}$ can also be non-continuous near the origin.

\noindent \textbf{Acknowledgment.}
The authors  sincerely thank the editor and  the reviewers for their careful reading and insightful comments,  which have clarified ambiguities, corrected errors, and improved the clarity of this paper. They are also grateful to Yupei Huang and Ovidiu Savin for their helpful suggestions. 

Guo is partially supported by the NSFC-12501145,  the Natural Science Foundation of Shanghai (No. 25ZR1402207), the Postdoctoral Fellowship Program of CPSF (No. GZC20252004), and  the China Postdoctoral Science Foundation (No. 2025T180838 and 2025M773061). 
Li is partially supported by NSFC-W2531006, NSFC-12250710674, NSFC-12031012 and the Institute of Modern Analysis-A Frontier Research Center of Shanghai. Zhang is partially supported by NSFC-12526202, and NSFC-12141105.

\noindent \textbf{Conflict of interest.} The authors do not have any possible conflicts of interest.

\noindent \textbf{Data availability statement.}
Data sharing is not applicable to this article, as no data sets were generated or analyzed during the current study.

\end{document}